\documentclass{article}
\usepackage[utf8]{inputenc}
\usepackage{amsfonts}
\usepackage{amsmath,amsthm}
\allowdisplaybreaks
\usepackage{amscd}
\usepackage[all,cmtip]{xy}
\usepackage{amssymb}
\usepackage{geometry}
\usepackage{fancyhdr}
\usepackage{lipsum}
\usepackage{mathtools}
\usepackage{relsize}
\usepackage{physics}
\usepackage{braket}
\usepackage{bm}
\usepackage{esint}
\usepackage{enumitem}
\usepackage{mathrsfs}
\usepackage{xcolor}
\usepackage{bbold}
\usepackage{comment}
\usepackage[
backend=biber,
style=alphabetic,
maxbibnames=99
]{biblatex} %Imports biblatex package
\usepackage{hyperref}
\hypersetup{
  colorlinks   = true,
  citecolor    = blue,
  linkcolor    = red
}

\newtheorem{theorem}{Theorem}[section]

\newtheorem{definition}[theorem]{Definition}

\newtheorem{lemma}[theorem]{Lemma}

\newtheorem{proposition}[theorem]{Proposition}
\newtheorem{remark}[theorem]{Remark}

\def\text{\mbox}
\def\bi{\begin{itemize}}
\def\ei{\end{itemize}}
\def\bem{\begin{emuerate}}
\def\eem{\end{enumerate}}
\def\beq{\begin{eqnarray*}}
\def\eeq{\end{eqnarray*}}

\def\ti{\textit}

\def\N{\mathbb{N}} %naturals
\def\R{\mathbb{R}} %reals
\def\C{\mathbb{C}} %complexes
\def\P{\mathbb{P}} %probability
\def\E{\mathbb{E}} %expectation
\def\Var{\operatorname{Var}}

\def\1{\mathbb{1}}

\newcommand*\Laplace{\mathop{}\!\mathbin\bigtriangleup}%Laplacian Operator
\newcommand*\DAlambert{\mathop{}\!\mathbin\Box}%D'Almbertian Operator

\title{Propagation of Singularities for the Damped Stochastic Klein-Gordon Equation}
\author{Hongyi Chen \footnote{Department of Mathematics, Aarhus University, Aarhus, Denmark. Email: hchen77@math.au.dk} \and Cheuk Yin Lee\footnote{School of Science and Engineering, The Chinese University of Hong Kong (Shenzhen), Longgang, Shenzhen, Guangdong, 518172, P.R. China. Email: leecheukyin@cuhk.edu.cn}}

\begin{document}

\maketitle

\begin{abstract}
    \noindent For the $1+1$ dimensional damped stochastic Klein-Gordon equation driven by additive space-time white noise, we show that random singularities associated with the law of the iterated logarithm exist and propagate in the same way as the stochastic wave equation without damping. 
    The proof can be reduced to the critically damped case, though many details of the proof still differ from the wave equation case.
    Our result provides evidence for possible connections to microlocal analysis, i.e., the exact regularity and singularities described in this paper should admit wavefront set type descriptions whose propagation is determined by the highest order terms of the linear operator.\\
    % Despite the results being exactly the same as those of the wave equation, our proofs are significantly different than the proofs for the wave equation. Miraculously, proving our results for the critically damped equation implies them for the general equation, which significantly simplifies the problem. Even after this simplification, many important parts of the proof are significantly different than existing proofs for the wave equation.
\end{abstract}


% \tableofcontents

\section{Introduction}

We study the regularity of the one-dimensional damped stochastic Klein-Gordon equation \begin{equation}\label{eq:DSKG}
    (\DAlambert+a\partial_t+m^2 )u(t,x)=\dot
    W(t,x), \quad t>0, x \in \R,
\end{equation} 
interpreted in the sense of \cite{Walsh86} subject to initial conditions $u(0,x)=\partial_tu(0,x)=0$.
Here $a,m\in \R$ are known from the physics literature as the damping constant and mass, $\DAlambert=\partial_{tt}-\partial_{xx}$ is the d'Alembertian in one space dimension, and $\dot W$ is Gaussian space-time white noise defined on a complete probability space $(\Omega, \mathscr{F}, \P)$.
Let $\{\mathscr{F}_t\}_{t\ge 0}$ denote the filtration generated by the noise $\dot{W}$. For readers interested in the physical motivation for the Klein-Gordon equation as well as the addition of the damping term and the noisy forcing term, we refer to \cite{brzezniak2005stochastic,song2020stochastic} and the references therein.
% or spatially homogeneous noise with covariance
% \[
%     \E[\dot{W}(t,x)\dot{W}(s,y)] = \delta(t-s) |x-y|^{-\beta},
% \]
% where $0<\beta<1$.
\begin{comment}
{\color{red}(These must be replaced)\\
We wish to show results similar to those proven in \cite{Walsh82,CN88,LeeXiao22}.
The operator $\DAlambert+a\partial_t+m^2$ does not have a simple fundamental solution, but has the same principal symbol as $\DAlambert$. One of the first fundamental results of microlocal analysis is that the linear equations $$(\DAlambert+a\partial_t+m^2 )u(t,x)=f(t,x)\text{ and }  \DAlambert u(t,x)=f(t,x)$$ propagate singularities  (i.e. points where the solution $u(t,x)$ is not smooth) identically \cite{Hormander78}. We believe that the same behavior holds for stochastic singularities, defined by Walsh as random points where the LIL fails \cite{Walsh82}. At this time, we do not know how to characterize these stochastic singularities in the language of microlocal analysis, so we must resort to more elementary methods.
}
\end{comment}
Our main results are the following Theorems.
\begin{theorem}[Law of the iterated logarithm]\label{th:lil}
    For any fixed $t_0\ge0$ and $x_0 \in \R$, there exists a constant $0<K_1<\infty$ depending on $(t_0, x_0)$ such that
    \begin{align}\label{E:lil}
        \limsup_{h \to 0^+} \frac{|u(t_0+h,x_0+ h)-u(t_0,x_0)|}{\sqrt{h \log\log(1/h)}} = K_1 \quad \text{a.s.}
    \end{align}
\end{theorem}

\begin{theorem}[Modulus of continuity]\label{th:mc}
    Fix $w_0\ge0$ and $0<a_1<a_2$, and consider the line segment
    \begin{align}\label{line:I}
        I = \left\{ (t,x) \in \R_+ \times \R : \tfrac{t- x}{\sqrt{2}} = w_0, \tfrac{t+ x}{\sqrt{2}} \in [a_1,a_2] \right\}.
    \end{align}
    Then, there exist constants $0<C_1<C_2<\infty$ depending on $(w_0,a_1,a_2)$ such that for every sub-interval $J$ of $I$, there exists a constant $K_J \in [C_1, C_2]$ such that
    \begin{align}
     \label{E:mc}   \limsup_{h\to0^+}\sup_{(t,x)\in J}\frac{|u(t+h,x+ h)-u(t,x)|}{\sqrt{h \log(1/h)}} = K_J \quad \text{a.s.}
    \end{align}
\end{theorem}

Theorem \ref{th:lil} states that at a fixed space-time point, the increments (along a characteristic direction) are locally of order $\sqrt{h\log\log(1/h)}$, while Theorem \ref{th:mc} asserts that the uniform modulus of continuity for the increments are of a larger order, at a logarithmic level.
Therefore, Theorems \ref{th:lil} and \ref{th:mc} together suggest that there exist random space-time points at which local increments are larger than those at a fixed point.
Theorem \ref{th:sing} below justifies the existence of these random singularities and shows that they propagate along the other characteristic direction.

One of the ingredients for proving Theorem \ref{th:sing} is the following simultaneous upper bound for the law of the iterated logarithm.
For notational convenience, we introduce a change of coordinates:
\begin{align}\label{w,z}
    (w,z) = \left(\tfrac{t-x}{\sqrt{2}}, \tfrac{t+x}{\sqrt{2}}\right), \qquad (t,x) = \left(\tfrac{w+z}{\sqrt{2}}, \tfrac{-w+z}{\sqrt{2}}\right).
\end{align}
The $w$ and $z$ directions are the characteristic directions of the wave equation.

\begin{theorem}\label{th:sim:lil}
    Fix $N>0$ and $0<a_1<a_2$. Then, there exists a constant $0<K_2<\infty$ such that for all fixed $z_0 \in [a_1,a_2]$,
    \begin{align}\label{E:sim:lil}
        \P\left\{ \limsup_{h\to0^+} \frac{|u\big(\frac{w+z_0}{\sqrt{2}}+h, \frac{-w+z_0}{\sqrt{2}}+h\big)-u\big(\frac{w+z_0}{\sqrt{2}}, \frac{-w+z_0}{\sqrt{2}}\big)|}{\sqrt{h\log\log(1/h)}} \le K_2 \text{ for all $w \in [0,N]$} \right\} = 1.
    \end{align}
\end{theorem}

\begin{theorem}[Existence and propagation of singularities]\label{th:sing}
    Fix $w_0\ge0$, $0<z_0<z_0'$. Let $t_0=w_0/\sqrt{2}$. Then:
    \begin{enumerate}
        \item[(i)] There exists an $\mathscr{F}_{t_0}$-measurable random variable $Z=Z(\omega)$ such that $z_0 \le Z \le z_0'$ a.s. and
    \begin{align}
        \limsup_{h\to0^+} \frac{|u\big(\frac{w_0+Z}{\sqrt{2}}+h,\frac{-w_0+Z}{\sqrt{2}}+h\big)-u\big(\frac{w_0+Z}{\sqrt{2}},\frac{-w_0+Z}{\sqrt{2}}\big)|}{\sqrt{h \log\log(1/h)}} = +\infty \quad \text{a.s.}
    \end{align}
        \item[(ii)] If $Z$ is any such random variable, then
        \begin{align}\label{E:propagation}
            \P\left\{\limsup_{h\to0^+} \frac{|u\big(\frac{w+Z}{\sqrt{2}}+h,\frac{-w+Z}{\sqrt{2}}+h\big)-u\big(\frac{w+Z}{\sqrt{2}},\frac{-w+Z}{\sqrt{2}}\big)|}{\sqrt{h \log\log(1/h)}} = +\infty \text{ for all $w>w_0$}\right\}=1.
        \end{align}
    \end{enumerate}
\end{theorem}
\begin{comment}
\begin{remark}\label{cor: >=1/2 wavefront set}
    The statements of Theorem \ref{th:lil} and \ref{th:mc} as well as part (i) of Theorem \ref{th:sing} can be extended to statements about increments in all directions as opposed to only the characteristic directions. To not make the paper unnecessarily complicated, we chose to avoid pursuing these extensions. These extensions of Theorems \ref{th:lil}--\ref{th:mc} together imply that with probability 1, the $B^s_{\infty,\infty}-$Besov wavefront set \cite{BesovWS} of $u$ is empty if $s<\frac{1}{2}$ and is equal to $(\R_+\times\R)\times (\R^2\setminus \{(0,0)\})$ if $s\geq \frac{1}{2}$. With some extra work, the same statements can be said about the more familiar $H^s-$Sobolev wavefront sets \cite[Definition 7.35, Proposition 7.36]{hintz2025introduction} of $u$ for the same ranges of $s$. We will discuss the relevance of this in the literature review.
\end{remark}
\end{comment}

\begin{remark}
    In general one can define $\DAlambert=b^2\partial_{tt}-c^2\partial_{xx}$, but a simple change of coordinates would allow one to produce the above results (with appropriately tilted characteristic lines) with the same calculations. In particular, what one would see is that singularities where the increments are observed in one characteristic direction would be propagated forward in time in the other characteristic direction. If one studies the proof of Theorem \ref{th:sim:lil} and the propagation part of \ref{th:sing}, one will discover that the proofs would fail if one tries to repeat it for any other pair of unit vectors that are not the characteristic pair. This is also predictable from the known behavior of hyperbolic equations, and we will discuss more of its relevance in the literature review. We also note that all of our results can apply to the excited stochastic Klein-Gordon equation $(a<0)$, but with extra dependence on $T$ the time horizon in all of the constants in the estimates.
\end{remark}

\subsection{Literature on LIL and L\'evy's Modulus of Continuity}

The (Khinchin) law of the iterated logarithm (LIL, for short) was first stated by Khinchin in 1924 \cite{khintchine1924satz} as a description of the fluctuations of a symmetric random walk. Since Brownian motion is the limit of scaled random walks, the LIL for Brownian motion \cite{khintchine1933asymptotische} can be interpreted as both regularity and large time fluctuation statements. We will concentrate on the regularity interpretation, for which we note that it has since been generalized to stochastic differential equations and Brownian sheets \cite{Zim,Orey-Pruitt}. Importantly, local versions of the LIL with a different power of $h$ in the denominator of \eqref{E:lil} has been proven for fractional Brownian Motion and fractional Brownian sheets \cite{Orey72,Ayache-Xiao}, which can be seen as a reflection of their regularity being higher or lower than Brownian motion. All of these have also been generalized to Gaussian processes, Gaussian random fields, and stochastic partial differential equations (SPDEs) driven by appropriate objects \cite{Marcus-Rosen,MWX,LeeXiao22,Hu-Lee}.

Another regularity property of Brownian motions, L\'evy's modulus of continuity (MC) \cite{Levy}, is apparently at odds with the LIL since it divides by a larger denominator compared to the statement of LIL. This has also been generalized to Gaussian processes and SPDEs \cite{Marcus-Rosen,MWX,LX19,LX23,Hu-Lee}. Notably, another way the LIL and MC are different is that the LIL fixes a point in parameter space while MC allows variation of the base point in parameter space. This indicates the existence of random points at which local increments are exceptionally large and fail to satisfy LIL. We will call such points LIL-singularities. 
Exceptional increments were first studied for the Brownian motion by Orey and Taylor \cite{Orey-Taylor} who showed that the set of LIL-singularities is dense, uncountable, and has Hausdorff dimension 1, and also studied different types of singularities.
Similar results are known for Brownian sheets, general Gaussian processes, and SPDEs \cite{Walsh82,Walsh86,Kh-Shi,KPX,Huang-Kh}.
% The discovery and further investigations of singularities for other Gaussian processes and SPDEs have been a fruitful area of study .
In particular, Walsh \cite{Walsh82} studied the LIL-singularities of a 2-parameter Brownian sheet and proved that singularities along an axis direction propagate perpendicularly along the other axis direction.
Walsh modified his argument to show that LIL-singularities of stochastic wave equation propagate along characteristic directions \cite{Walsh86}.
These results have been extended to semi-fractional Brownian sheets \cite{Blath-Martin} and more general stochastic wave equations \cite{CN88, LeeXiao22}.

\subsection{Microlocal Interpretation and Heuristics}
Theorem \ref{th:sim:lil} and the propagation part of Theorem \ref{th:sing} admit natural interpretations in the language of microlocal analysis\cite{sato1970regularity,hormander1971fourier,hormander1973existence,vasy2011microlocalanalysisasymptoticallyhyperbolic,hintz2025introduction}. For the differential operator $P=\DAlambert+a\partial_t+m^2$, its principal symbol (a function on $(\R_t\times\R_x)\times (\R^2\setminus\set{0})$, interpreted as the nonzero cotangent space of $\R_t\times\R_x$) is \[p(t,x;\sigma,\xi)=-\sigma^2+\xi^2,\] which corresponds only to the leading order term $\DAlambert$. The \textit{Hamiltonian vector field} of $p$ is \[H_p=-2\sigma\partial_t+2\xi\partial_x.\]
The integral curves of $H_p$ projected into physical spacetime are exactly the characteristic lines, which regularity (Theorem \ref{th:sim:lil}) and singularity (Theorem \ref{th:sing} part (ii)) propagate along.

The principle that our results seek to explore is that propagation of singularities is determined entirely by the highest order terms. In microlocal analysis, this is made precise using the language of wavefront sets in a seminal result of H\"ormander \cite{Hormander78}. Informally, a point $(x,\xi)\in \R^{n}\times(\R^{n}\setminus \{0\})$ (viewed as an element the nonzero cotangent bundle of $\R^n$) is an element of the (smooth) wavefront set of some Schwartz distribution $\phi$ if $\phi$ is not smooth in the direction $\xi$ at $x$. For a real hyperbolic operator $L$ an equation of the form $Lu=\phi$, \cite{Hormander78} tells us that the wavefront set of $u$ can be obtained by propagating the wavefront set of $\phi$ along the null-bicharacteristics generated by the Hamiltonian flow of $L$, which are determined entirely by the principal symbol. By symmetry, a corresponding result applies to the complement of the wavefront set, showing that regularities propagate the same way. Definitions of wavefront sets of Sobolev and Besov type (including for negative regularities), as well as proofs of the analogue of \cite{Hormander78} for these definitions, have been discovered since then \cite{sjostrand1982singularites,dencker1982propagation,BesovWS} (see \cite[Chapters 6--8]{hintz2025introduction} for a modern exposition on the precise definition of wavefront sets, including Sobolev wavefront sets, and the result of \cite{Hormander78}). Due to the diffeomorphism invariance of wavefront sets, this theory also extends to smooth manifolds, which has been useful in mathematical general relativity \cite{hintz2018global,hintz2021normally,hafner2021linear}. 

At present,  we do not have a wavefront set framework adapted to the singularities and regularities considered in this paper. Consequently, we lack the foundation on which H\"ormander’s propagation theorem could be formulated for our objects of study. Nevertheless, Theorems \ref{th:sim:lil} and \ref{th:sing} are consistent with the intuition it gives: the singularities (with directions) studied here propagate along the bicharacteristics of the Hamiltonian flow. This suggests that an appropriate stochastic analogue of the wavefront set, once developed, should also exhibit propagation governed by the principal symbol as described by \cite{Hormander78}. We further conjecture that one consequence of extending the microlocal framework to this setting is that the Edwards-Wilkinson equation \begin{equation}\label{eq: EW}
    (\partial_t-\Laplace)v(t,x)=\dot{W}(t,x)
\end{equation}
will not propagate its singularities, which are known to exist in $1+1$ dimensions \cite{Hu-Lee}. Instead, the well-known smoothing properties of the heat semigroup should instantaneously regularize all such singularities for positive times. Our main results provide heuristics for this in $1+1$ dimensions using the Smolchowski-Kramers approximation \begin{equation}
    (\mu^2\partial_{tt}+\partial_t-\partial_{xx})v_\mu(t,x)=\dot{W}(t,x),\quad \mu\in (0,1]
\end{equation}
where one will see \eqref{eq: EW} as $\mu\downarrow0$ (we refer interested readers to \cite{cerrai2006smoluchowski,cerrai2016smoluchowski,cerrai2023small} for the stability of this approximation and related topics). For $\mu>0$, one can easily adopt the contents of this paper to see that Theorems \ref{th:lil}--\ref{th:sing} all apply to $v_\mu$ with appropriately tilted characteristic lines. The effect of tilting the characteristic lines as $\mu$ gets smaller is that for any singularity of $v_\mu$ at some point $(t_0,x_0)$, the propagated singularity of part (ii) of Theorem \ref{th:sing} is at $(t_0+\varepsilon,x_0+\frac{\varepsilon}{\mu})$ at time $t_0+\varepsilon$, which notably goes to infinity as $\mu\downarrow0$. This is consistent with the PDE heuristic that hyperbolic operators have finite speed of propagation while parabolic operators have infinite speed of propagation, so $\mathcal{L}_\mu:=\mu^2\partial_{tt}+\partial_t-\partial_{xx}$ indexes a family of hyperbolic operators whose speed of propagation goes to infinity as $\mu\downarrow0$.

We now discuss the current wavefront set literature to make clear what is still needed to attack the above conjecture. The following definition is taken from \cite[Definition 23]{BesovWS}. Let $\mathcal{F}$ be the Fourier transform on $\R^d$, $B(0,1)$ be the unit ball on $\R^d$, and $u$ be a Schwartz distribution on $\R^d$. For $\alpha\in \R$, we say that $(x_0,\xi_0)\in \R^d_x\times(\R^d_\xi\setminus\set{0})$ is \ti{not in the $B^{\alpha}_{\infty,\infty}$  wavefront set} of $u$ if there exists an open $K\supset x_0$ and an open conic $\Gamma\supset \xi_0$, such that for all $\phi\in C_c^\infty(\R^d)$ with $\phi(x)\neq 0$ for all $x\in K$, $\kappa\in C_c^\infty(B(0,1))$ with $\partial^\beta\kappa(0)=0$ if $|\beta|<\lfloor\alpha\rfloor$, $\tilde{\kappa}\in C_c^\infty(A)$, where $A$ is an annulus in $\R^d$, we have \begin{align*}
    &\abs{\int_\Gamma \mathcal{F}[\phi u](\xi)\kappa(\xi)e^{ix\cdot \xi} d\xi}\lesssim 1,\quad \text{and}\\
    &\abs{\int_\Gamma \mathcal{F}[\phi u](\xi)\tilde{\kappa}(\lambda\xi)e^{ix\cdot \xi} d\xi}\lesssim \lambda^\alpha \quad \text{for all $\lambda\in (0,1]$}.
\end{align*}
The set of $(x_0,\xi_0)\in\R^d_x\times(\R^d_\xi\setminus\set{0})$ not satisfying the above description is the $B^{\alpha}_{\infty,\infty}$ wavefront set of $u$, which we denote as $\mathrm{WF}^\alpha(u)$. Proposition 25 of \cite{BesovWS} then tells us that $u\in B^{\alpha,\mathrm{loc}}_{\infty,\infty}\iff \mathrm{WF}^\alpha(u)=\emptyset$, where $B^{\alpha,\mathrm{loc}}_{\infty,\infty}$ is the local $\infty,\infty-$Besov space of order $\alpha$. We will not define this independently here, but the discussion following \cite[Definition 2.68]{BCD11} tells us that for $\alpha\in (0,1)$, it is equivalent to $C^{\alpha,\mathrm{loc}}(\R^d)$. Thus, classical notions of regularity admit natural descriptions using wavefront sets. We believe that with some extra work (figuring out how $\log(\lambda)$ properly appears in the right hand side of the second estimate), a corresponding definition can be formulated for the modulus of continuity. The classical L\'evy modulus of continuity for Brownian motion would imply that Brownian motion paths satisfy this definition in dimension 1 with probability 1. As stated, Theorem \ref{th:mc} would not imply the solution of \eqref{eq:DSKG} satisfies such a definition in dimension 2, since it only deals with characteristic increments, but we believe that the conclusion can be deduced from the machinery we used for its proof. To investigate propagation in the language of wavefront sets, one would need to have the LIL difference quotient at a base point with a given direction (or lack thereof) and have an equivalent formulation in terms of wavefront sets. What would be different about a corresponding definition for LIL regularity or singularity in comparison to the Besov-H\"older definition given above is that it is a \textit{strictly pointwise} phenomenon, meaning that information about ``uniform behavior in a neighborhood of the base point" should not be used for such a definition. 
Currently, to our knowledge, such a definition does not exist in the literature (except possibly in $d=1$, where one simply excludes the other direction, but nothing really interesting happens here as far as propagation is concerned).

\begin{comment}
Notations: 
$\R_+ = (0,\infty)$; 
$\N = \{0,1,2,\dots\}$; 
$\N_+ = \{1,2,\dots\}$; 
``$f(x) \lesssim g(x)$'' means that there exists $C<\infty$ such that $f(x) \le Cg(x)$ for all $x$;
``$f(x)\propto g(x)$'' means that there exists $0<C<\infty$ such that $f(x) = Cg(x)$ for all $x$
\end{comment}

\subsection{Paper organization, proof overview, and notation}
The rest of the paper will be organized as follows. In Section \ref{s:pre}, we state our Fourier transform convention, then prove pointwise and increment second moment bounds for our solution $u$, and hence establish the H\"older regularities of $u$. In Section \ref{s:reduction}, we show that proving any of our main theorems in the case of the critically damped equation actually implies the same theorem for all damped equations. The rest of the paper is then dedicated to proving the main theorems for the critically damped equations.
This reduction trick simplifies the proofs significantly. %Notably, the proofs of Theorem \ref{th:sim:lil} and the propagation part of Theorem \ref{th:sing} are significantly different compared to existing proofs of analogous results for the wave equation, despite the simplification of section 3. In fact, one can see directly in the propagation proof that damping does not erase the singularity.\\
In Section \ref{s:4}, we prove Theorems \ref{th:lil} and \ref{th:mc}.
Finally, we prove Theorems \ref{th:sim:lil} and \ref{th:sing} in Sections \ref{s:5} and \ref{s:6}, respectively.

Here is an outline of the proof strategies for the main theorems.
The LIL and L\'evy-type modulus of continuity results (Theorems \ref{th:lil} and \ref{th:mc}) follow from standard Gaussian arguments. In particular, we prove these theorems using the framework developed in \cite{LX23}, applicable to a large class of Gaussian random fields with non-stationary increments.
In particular, these results rely on the variance bounds for the increments given by Lemmas \ref{lem: time increment bound} and \ref{lem: space increment bound}, and the harmonizable representation (see Lemma \ref{lemma: Polarity 8.3}). Moreover, the modulus of continuity result hinges on the strong local nondeterminism property established in Lemma \ref{lem:SLND}.

As in \cite{LeeXiao22}, the proof of the simultaneous LIL (Theorem \ref{th:sim:lil}) is based on a second moment estimate for the rectangular increments of the solution, but the computations are different from those in \cite{LeeXiao22} due to the presence of damping.
The proof of Theorem \ref{th:sing} follow a strategy that is similar to those in existing work on the undamped wave equation, but the details are significantly different due to the damping term. 
Without damping, the solution can be decomposed into two independent components $u_1+u_2$, where $u_1$ has Brownian increments and $u_2$ is a solution (in law) to the stochastic wave equation, hence the L\'evy modulus of continuity ensures that $u_1$ has a random singularity which persists for all future times along a characteristic direction due to the simultaneous LIL of $u_2$.
Our proof differs from the undamped wave equation case in the following way. Under critical damping, $u_2$ still solves the original equation, while the increments of $u_1$ can be expressed as damped Brownian increments plus a smaller-order term.
In finite time, the damping effect does not remove the Brownian singularity, nor does the smaller-order term, nor does $u_2$ which satisfies a simultaneous LIL, hence the singularity propagates for all future times.

% The most significant differences compared to existing work on the wave equation are in Sections \ref{s:5} and \ref{s:6}, where we prove Theorems \ref{th:sim:lil} and \ref{th:sing}, respectively. 
% Specifically, the wave equation version of these portions relied heavily on the exact formula of the fundamental solution of $\DAlambert$, which we cannot use. (You should say more...)\\

We list some notation we will employ for the rest of the paper. \begin{itemize}
    \item Sets: $\R_+ = (0,\infty)$, $\N = \{0,1,2,\dots\}$, $\N_+ = \{1,2,\dots\}$
    \item $a\wedge b = \min\{a,b\}$; $\1_A$ denotes indicator function of a set $A$
    \item $f(x)\propto g(x)$ means that there exists $0<C<\infty$ such that $f(x) = Cg(x)$ for all $x$, and $f(x) \lesssim g(x)$ means that there exists $C<\infty$ such that $f(x) \le Cg(x)$ for all $x$. In one calculation, we may write these in succession without mentioning that the constants change from line to line, and we will only note dependence of the constants on certain parameters if we deem it relevant to do so.
\end{itemize}
% {\color{magenta}This should be checked.
% \begin{remark}[Extension to Nonlinear Multiplicative Noise]
%     After the first version of this manuscript was released, \cite{DSKG_linearization} used the techniques we introduced (specifically, the reduction to critical damping in Section 3) to show that the space-time increments of the nonlinear equation $$(\DAlambert+a\partial_t+m^2 )U(t,x)=
%     F(U(t,x))\dot W(t,x), \quad t>0, x \in \R,\quad U(0,x),\partial_xU(0,x)\equiv 0$$
%     linearize and behave like the increments of $u$. More precisely, as the size of the increments $\varepsilon\downarrow0$, they show that the difference between the increments of $U$ and the increments of $u$ are of order $\varepsilon^{\frac{3}{2}}$. As a consequence, our results also apply to $U$.
% \end{remark}
% }
\section{Preliminaries}\label{s:pre}
For notational clarity, we choose a Fourier transform convention and state some basic theorems in it.
\begin{definition}
    For $f\in L^2(\R)$, define its Fourier transform by \[\mathcal{F}[f](\xi)=\hat{f}(\xi):=\int_\R f(x)e^{-ix\xi}dx.\]
    The following are standard results.
    \begin{enumerate}
        \item Translation: $\mathcal{F}[f(\cdot-a)](\xi)=e^{-ia\xi}\hat{f}(\xi)$.
        \item Differentiation: $\mathcal{F}\left[\dv[k]{f}{x}\right]=(i\xi)^{k}\hat{f}(\xi).$
        \item Parseval-Plancherel identity: $\braket{f,g}_{L^2(\R)}=\frac{1}{2\pi}\braket{\hat{f},\hat{g}}_{L^2(\R)}$.
        \item Convolution: define $f*g(x):=\int_\R f(x-y)g(y) dy$. Then $\mathcal{F}[f*g](\xi)=\hat{f}(\xi)\hat{g}(\xi).$\\
    \end{enumerate}
\end{definition}

% \subsection{The fundamental solution in Fourier space}
The solution to \eqref{eq:DSKG} is \begin{equation}\label{eq:mild sol}
    u(t,x)=\int_0^t \int_{-\infty}^{\infty} G(t-s,x-y) W(ds,dy)
\end{equation}
where $G$ denotes the fundamental solution to $\DAlambert+a\partial_t+m^2$ and the integral on the right hand side is understood in the sense of Walsh \cite{Walsh86}. 
The existence and regularity of $G$ is an elementary PDE exercise. Taking the Fourier transform in space, one finds that $\hat{G}(t,\xi)$, the Fourier transform of $G(t,\cdot)$, satisfies \begin{align}\begin{split}\label{Fourier G}
    \hat{G}(t,\xi)=&\,\frac{e^{-at/2}\sin\left(t\sqrt{\xi^2+m^2-\frac{a^2}{4}}\right)}{\sqrt{\xi^2+m^2-\frac{a^2}{4}}}\1_{\set{\xi^2>\frac
    {a^2}{4}-m^2}}+te^{-\frac{at}{2}}\1_{\set{\xi^2=\frac{a^2}{4}-m^2}}\\
    &+\frac{e^{-at/2}\sinh{\left(t\sqrt{\frac{a^2}{4}-m^2-\xi^2}\right)}}{\sqrt{\frac{a^2}{4}-m^2-\xi^2}}\1_{\set{\xi^2<\frac
    {a^2}{4}-m^2}}.
\end{split}\end{align}
Observe that the small frequency regime exists if and only if $\frac{a^2}{4}-m^2> 0$.
% , but even if it fails to exist, we cannot repeat what was done in \cite{CN88} because $G$ is non-elementary. 
We see through the above expression explicitly how damping and mass terms affect the solution by comparing it to the wave fundamental solution \cite{evansPDE}. Note that $G$ is non-elementary except in the critial damping case $m^2 = a^2/4$ (see \eqref{G: critical} below).
% , where
% \begin{align}\label{G:critical}
%     \hat{G}(t,\xi) = \frac{e^{-at/2}\sin(t|\xi|)}{|\xi|}
%     \quad \text{and hence} \quad
%     G(t,x) = \frac{e^{-at/2}}{2} \1_{\{|x|< t\}}.
% \end{align}
The following estimates will be useful.
\begin{lemma}\label{lem:variance bound}
    Let $u$ be defined as in \eqref{eq:mild sol}. 
    Fix $0<T<\infty$.
    Then
    $$\E[u(t,x)^2] \lesssim t^2$$
    uniformly for all $t\in[0,T]$ and $x \in \R$, where the implicit constant depends only on $T,a,m$.
\end{lemma}
\begin{proof}
    Using It\^{o}--Walsh isometry, the change of variables $s\mapsto t-s$, $y\mapsto x-y$, and Plancherel's theorem, we can write
    \begin{align*}
        \E[u(t,x)^2]
        =\int_0^t \int_{-\infty}^\infty |G(t-s,x-y)|^2 dyds
        =\int_0^t \int_{-\infty}^\infty |G(s,y)|^2 dyds
        % \propto& \int_0^t \int_{-\infty}^{\infty}\mathcal{F}[G(s,x-\cdot)](\xi) \overline{\mathcal{F}[G(s,x-\cdot)](\xi)} d\xi ds\\
        % =&\int_0^t \int_{-\infty}^{\infty} e^{- ix\xi}e^{ix\xi}G(s,\xi)^2  d\xi ds\\
        \propto\int_0^t \int_{-\infty}^{\infty} |\hat G(s,\xi)|^2d\xi ds.
    \end{align*}
    Then, by \eqref{Fourier G},
    \begin{align*}
        \E[u(t,x)^2]
        \propto&\int_0^t \int_{-\infty}^{\infty} \left[e^{-as}\frac{\sin^2\left(s\sqrt{\xi^2+m^2-\frac{a^2}{4}}\right)}{\xi^2+m^2-\frac{a^2}{4}}\1_{\set{\xi^2>\frac
        {a^2}{4}-m^2}} \right.\\
        &\qquad \left.+e^{-as}\frac{\sinh^2{\left(s\sqrt{\frac{a^2}{4}-m^2-\xi^2}\right)}}{\frac{a^2}{4}-m^2-\xi^2}\1_{\set{\xi^2<\frac
        {a^2}{4}-m^2}}\right]d\xi ds\\
        = & \, I_1+I_2,
    \end{align*}
    where 
    % $K^2=(m^2-\frac{a^2}{4})\wedge 0$,
    \begin{align*}
        I_1:=\int_0^t \int_{|\xi|<K} e^{-as}\frac{\sinh^2{\left(s\sqrt{K^2-\xi^2}\right)}}{K^2-\xi^2}d\xi ds,\quad I_2:= \int_0^t \int_{|\xi|>K} e^{-as} \frac{\sin^2(s\sqrt{\xi^2-K^2})}{\xi^2-K^2} d\xi ds,
    \end{align*}
    and $K^2 = (\frac{a^2}{4}-m^2) \vee 0$.
    For $I_1$, we have
    \begin{align*}
        I_1=& \int_0^t \int_{\abs{\xi}<K}s^2 e^{-as}\frac{\sinh^2{\left(s\sqrt{K^2-\xi^2}\right)}}{s^2(K^2-\xi^2)}d\xi ds\\
        \lesssim& \int_0^t s^2e^{-as}ds
        \lesssim t^3 \lesssim t^2
        % =e^{-as}\left(-\frac{s^2}{a}-\frac{2s}{a^2}-\frac{2}{a^3}\right)\eval_{0}^{t}\lesssim t^2.
    \end{align*}
    uniformly for all $t \in [0,T]$ and $x \in \R$, 
    where the first $\lesssim$ above used the fact that $\sinh(x)/x$ is uniformly bounded over any compact interval $I\subset \R$, and the implicit constants depend on $T,a,m$.

    For $I_2$, we choose and fix some constant $L>K$, say $L = K+1$, and further decompose $I_2=J_1+J_2,$ where \begin{align*}
    J_1:=\int_0^t\int_{K<\abs{\xi}<L} e^{-as} \frac{\sin^2(s\sqrt{\xi^2-K^2})}{\xi^2-K^2} d\xi ds,\quad J_2:=\int_0^t\int_{\abs{\xi}>L} e^{-as} \frac{\sin^2(s\sqrt{\xi^2-K^2})}{\xi^2-K^2} d\xi ds.
    \end{align*}
    We then use $\abs{\frac{\sin x}{x}}\leq 1$ and the same trick used for $I_1$ to obtain $J_1\lesssim t^2$.
    For $J_2$, we may use $|\sin{x}| \le 1 \wedge |x|$ and the fact that $\xi^2 - K^2 \gtrsim \xi^2$ for $|\xi| > L=K+1$ to obtain
    \begin{align*}
        J_2 &\lesssim \int_0^t \int_{|\xi|>L} e^{-as} \left( s^2 \wedge \frac{1}{\xi^2-K^2} \right) d\xi ds
        \lesssim \int_0^t \int_{|\xi|>L}  \left( s^2 \wedge \frac{1}{\xi^2} \right) d\xi ds
        \le \int_0^t \int_0^\infty \left( s^2 \wedge \frac{1}{\xi^2} \right) d\xi ds\\
        & \lesssim \int_0^t \left[ \int_0^{1/s} s^2 d\xi + \int_{1/s}^\infty \frac{1}{\xi^2} d\xi \right] ds
        \lesssim \int_0^t s \, ds \lesssim t^2.
    \end{align*}
    % For $J_2$, if $t\geq \frac
    % {1}{2K}$, we have
    % \[J_2\lesssim \int_0^t \int_{\abs{\xi}>L} \frac{e^{-as}}{\xi^2}d\xi ds\lesssim t \textcolor{purple}{\lesssim t^2}.\]
    % \textcolor{purple}{For $t<\frac{1}{2K}$, one can use $\abs{\frac{\sin x}{x}}\leq 1\wedge \frac{1}{x}$ and some calculus to obtain the matching $t^2$ bound.} 
    This concludes the proof for the second moment bound.
\end{proof}

\begin{lemma}\label{lem: time increment bound}
    Fix $0<T<\infty$.
    Then
    $$\E[|u(t,x)-u(s,x)|^2] \lesssim t(t-s)$$
    uniformly for $0\leq s<t\leq T$ and $x\in \R$, where the implicit constant depends only on $T,a,m$.
\end{lemma}

\begin{proof}
    We have \[u(t,x)-u(s,x)=\int_0^s \int_\R G(t-\tau,x-y)-G(s-\tau,x-y) W(d\tau,dy)+\int_s^t \int_\R G(t-\tau,x-y) W(d\tau,dy),\] giving us \begin{align*}
        \E[|u(t,x)-u(s,x)|^2]\propto&\int_0^s \int_\R \abs{\mathcal{F}[G(t-\tau,x-\cdot)-G(s-\tau,x-\cdot)](\xi)}^2d\xi d\tau\\
        &+\int_s^t \int_\R |\hat{G}(t-\tau,\xi)|^2 d\xi d\tau=:I_{1}+I_{2}.
    \end{align*}
    By a change of variable, the moment bound in Lemma \ref{lem:variance bound} gives $I_{2}\lesssim (t-s)^2 \le t(t-s)$.
    
    For $I_{1}$, again we use $K^2=(\frac{a^2}{4}-m^2)\vee 0$,
    choose and fix a constant $L>K$, say $L=K+1$,
    and apply the decomposition \[I_{1}=\int_0^s \int_\R (\cdots)=\int_0^s\int_{|\xi|<K}(\cdots)+\int_0^s \int_{K<|\xi|<L}(\cdots)+\int_0^s\int_{|\xi|>L}(\cdots)=:J_1+J_2+J_3.\]
    For  $J_1$, %\eqref{est:exp deriv} we have
    we use the following application of the triangle inequality
    \[\label{est: exp trieq}
    |e^{-x} f(t) - e^{-y} f(s)|
    \le |e^{-x}( f(t)-f(s))|
    + |(e^{-x}-e^{-y})f(s)|,
    \]
    together with the elementary inequality $(a+b)^2 \le 2(a^2+b^2)$, to deduce the following:
    \begin{align*}
        J_1\propto&\int_0^s\int_{|\xi|<K} \left[e^{-\frac{a(t-\tau)}{2}}\frac{\sinh\left((t-\tau)\sqrt{K^2-\xi^2}\right)}{\sqrt{K^2-\xi^2}}-e^{-\frac{a(s-\tau)}{2}}\frac{\sinh\left((s-\tau)\sqrt{K^2-\xi^2}\right)}{\sqrt{K^2-\xi^2}}\right]^2d\xi d\tau\\
        %=&\int_0^s\int_{|\xi|<K} \left[(t-\tau)^2e^{-\frac{a(t-\tau)}{2}}\frac{\sinh\left((t-\tau)\sqrt{K^2-\xi^2}\right)}{(t-\tau)^2(K^2-\xi^2)}-(s-\tau)^2e^{-\frac{a(s-\tau)}{2}}\frac{\sinh\left((s-\tau)\sqrt{K^2-\xi^2}\right)}{(s-\tau)^2(K^2-\xi^2)}\right]^2  d\xi d\tau \\
        \leq& \, 2\int_0^s\int_{\abs{\xi}<K} \left\{\Bigg[e^{-\frac{a}{2}(s-\tau)}\abs{\sinh((t-\tau)\sqrt{K^2-\xi^2})-\sinh((s-\tau)\sqrt{K^2-\xi^2})}\Bigg]^2 \right.\\
        & \left.+\Bigg[\sinh((t-\tau)\sqrt{K^2-\xi^2})\abs{e^{-\frac{a}{2}(t-\tau)}-e^{-\frac{a}{2}(s-\tau)}}\Bigg]^2\right\} \frac{d\xi d\tau}{K^2-\xi^2}\\
        % \leq&\, 2\int_0^s\int_{\abs{\xi}<K} \Bigg[e^{-\frac{a}{2}(s-\tau)}\sqrt{K^2-\xi^2}(t-s)\sup_{\tau\in [0,t-s]}\cosh(\tau\sqrt{K^2-\xi^2})\Bigg]^2\\
        % &\qquad +\Bigg[\sinh((t-\tau)\sqrt{K^2-\xi^2})(t-s)\frac{a}{2}\Bigg]^2 \frac{d\xi d\tau}{K^2-\xi^2}\\
        \lesssim& \int_0^s\int_{\abs{\xi}<K} \Bigg[\sqrt{K^2-\xi^2}(t-s)\cosh\left(T\sqrt{K^2-\xi^2}\right)+\sinh(T\sqrt{K^2-\xi^2})(t-s)\frac{a}{2}\Bigg]^2 \frac{d\xi d\tau}{K^2-\xi^2}\\
        =&\, s(t-s)^2 \int_{\abs{\xi}<K} \frac{\left[\sqrt{K^2-\xi^2}\cosh\left(T\sqrt{K^2-\xi^2}\right)+\sinh(T\sqrt{K^2-\xi^2})\frac{a}{2}\right]^2}{K^2-\xi^2} d\xi.
    \end{align*}
    The third inequality used the mean value theorem and the fact that both $\exp$ and $\cosh$ are increasing on $\R_+$. The singularities of the integrand at $\xi=0,\pm K$ are integrable by an elementary calculus argument.
    This yields $J_1 \lesssim s(t-s)^2 \lesssim t(t-s)$.
    
    Similarly, we have \begin{align*}
        J_2\propto&\int_0^s\int_{K<|\xi|<L} \left[e^{-\frac{a(t-\tau)}{2}}\frac{\sin\left((t-\tau)\sqrt{\xi^2-K^2}\right)}{\sqrt{K^2-\xi^2}}-e^{-\frac{a(s-\tau)}{2}}\frac{\sin\left((s-\tau)\sqrt{K^2-\xi^2}\right)}{\sqrt{\xi^2-K^2}}\right]^2d\xi d\tau\\
        %=&\int_0^s\int_{|\xi|<K} \left[(t-\tau)^2e^{-\frac{a(t-\tau)}{2}}\frac{\sinh\left((t-\tau)\sqrt{\xi^2-K^2}\right)}{(t-\tau)^2(\xi^2-K^2)}-(s-\tau)^2e^{-\frac{a(s-\tau)}{2}}\frac{\sinh\left((s-\tau)\sqrt{\xi^2-K^2}\right)}{(s-\tau)^2(\xi^2-K^2)}\right]^2  d\xi d\tau \\
        \leq& \, 2\int_0^s\int_{K<|\xi|<L} \left\{\Bigg[e^{-\frac{a}{2}(s-\tau)}\abs{\sin((t-\tau)\sqrt{\xi^2-K^2})-\sin((s-\tau)\sqrt{\xi^2-K^2})}\Bigg]^2 \right.\\
        & \left.+\Bigg[\sin((t-\tau)\sqrt{\xi^2-K^2})\abs{e^{-\frac{a}{2}(t-\tau)}-e^{-\frac{a}{2}(s-\tau)}}\Bigg]^2\right\} \frac{d\xi d\tau}{\xi^2-K^2}\\
        % \leq&\int_0^s\int_{K<|\xi|<L} \Bigg[e^{-\frac{a}{2}(s-\tau)}\sqrt{\xi^2-K^2}(t-s)+\sin((t-\tau)\sqrt{\xi^2-K^2})(t-s)\frac{a}{2}\Bigg]^2 \frac{d\xi d\tau}{\xi^2-K^2}\\
        \lesssim& \int_0^s\int_{K<|\xi|<L} \Bigg[\sqrt{\xi^2-K^2}(t-s)+\left|\sin((t-\tau)\sqrt{\xi^2-K^2})\right|(t-s)\frac{a}{2}\Bigg]^2 \frac{d\xi d\tau}{\xi^2-K^2}\\
        =&\, s(t-s)^2 \int_{K<\abs{\xi}<L} \frac{\left[\sqrt{\xi^2-K^2}+T\sqrt{\xi^2-K^2}\frac{a}{2}\right]^2}{\xi^2-K^2} d\xi \lesssim t(t-s).
    \end{align*}
    % Again the singularity at $K$ is integrable by elementary arguments.

    For $J_3$, we split it into two parts as we did above.
    For the first term, use the trigonometric identity $\sin(A) - \sin(B) = 2\cos(\frac{A+B}{2})\sin(\frac{A-B}{2})$ and bound the $\cos$ factor by 1.
    % \[
    % s(t-s)\int_{(t-s)L}^{\infty} \frac{\sin^2(\frac12\sqrt{\xi^2 - (t-s)^2 K^2})}{\xi^2 - (t-s)^2 K^2} d\xi \le Cs(t-s).
    % \]
    % \begin{align*}
    %     \int_{(t-s)L}^{\infty} \frac{\sin^2(\sqrt{\xi^2 - (t-s)^2 K^2})}{\xi^2 - (t-s)^2 K^2} d\xi
    %     &\le \int_{(t-s)L}^{\infty} \frac{1 \wedge (\xi^2 - (t-s)^2 K^2)}{\xi^2 - (t-s)^2 K^2} d\xi\\
    %     & \le \int_{(t-s)L}^{M} 1 d\xi + \int_M^\infty \frac{1}{\xi^2/2} d\xi
    % \end{align*}
    For the second term, bound $\sin$ by 1 and apply mean value theorem to the difference of the exponential functions.
    This yields the following:
    %We must do something different for the high-frequency term. 
    \begin{align*}
        J_3\propto& \int_0^s\int_{\abs{\xi}>L} \left[e^{-\frac{a(t-\tau)}{2}}\frac{\sin\left((t-\tau)\sqrt{\xi^2-K^2}\right)}{\sqrt{\xi^2-K^2}}-e^{-\frac{a(s-\tau)}{2}}\frac{\sin\left((s-\tau)\sqrt{\xi^2-K^2}\right)}{\sqrt{\xi^2-K^2}}\right]^2 d\xi d\tau\\
        \leq& \, 2\int_0^s\int_{|\xi|>L} \Bigg[e^{-\frac{a}{2}(s-\tau)}\abs{\sin((t-\tau)\sqrt{\xi^2-K^2})-\sin((s-\tau)\sqrt{\xi^2-K^2})}\Bigg]^2\\
        &+\Bigg[\sin((t-\tau)\sqrt{\xi^2-K^2})\abs{e^{-\frac{a}{2}(t-\tau)}-e^{-\frac{a}{2}(s-\tau)}}\Bigg]^2 \frac{d\xi d\tau}{\xi^2-K^2}\\
        % \leq& \int_0^s\int_{|\xi|>L} \Bigg[ 2\sin((t-s)\sqrt{\xi^2-K^2}) + \frac{a(t-s)}{2}\sin((t-\tau)\sqrt{\xi^2-K^2})\Bigg]^2\frac{d\xi d\tau}{\xi^2-K^2}\\
        \leq&\, 2\int_0^s\int_{|\xi|>L} \Bigg[ 4\sin^2((t-s)\sqrt{\xi^2-K^2}) + a^2(t-s)^2\sin^2((t-\tau)\sqrt{\xi^2-K^2})\Bigg]\frac{d\xi d\tau}{\xi^2-K^2}\\
        \lesssim& \int_0^s\int_L^\infty\frac{\sin^2((t-s)\sqrt{\xi^2-K^2})}{\xi^2-K^2} d\xi d\tau +(t-s)^2\int_0^s\int_L^\infty\frac{\sin^2((t-\tau)\sqrt{\xi^2-K^2})}{\xi^2-K^2} d\xi d\tau=:S_1+S_2.
        %\leq & \int_0^s\int_{\abs{\xi}>L} \frac{(t-s)^2e^{-a(t-s)}}{\xi^2-K^2}\sup_{\tau\in [0,t-s]}\left(-\frac{a}{2}\sin(\tau\sqrt{\xi^2-K^2})+\sqrt{\xi^2-K^2}\cos(\tau\sqrt{\xi^2-K^2})\right)^2  d\xi d\tau\\
        %{\color{red}??}\lesssim& (t-s)^2(1+e^{-at})
    \end{align*}
    Since $L>K$, we have $S_2\lesssim s(t-s)^2 \le t(t-s)$.
    For $S_1$, we use the change of variable $\xi\mapsto \xi/(t-s)$, giving us \begin{align*}
        S_1=&\, s(t-s)\int_{(t-s)L}^{\infty} \frac{\sin^2(\sqrt{\xi^2 - (t-s)^2 K^2})}{\xi^2 - (t-s)^2 K^2} d\xi\\
        \leq &\, t(t-s)\int_{(t-s)L}^{\infty} \frac{1 \wedge (\xi^2 - (t-s)^2 K^2)}{\xi^2 - (t-s)^2 K^2} d\xi\\
        \le&\, t(t-s) \left(\int_{(t-s)L}^{2TL} 1 d\xi + \int_{2TL}^\infty \frac{1}{\xi^2/2} d\xi\right) \lesssim t(t-s),
    \end{align*}
    where we have used the elementary fact that $\xi>2TL$ implies $\xi^2 - (t-s)^2K^2 \ge \xi^2 - T^2L^2 \ge \xi^2/2$ to obtain the second term in the second inequality.
    This finishes the proof of Lemma \ref{lem: time increment bound}.
\end{proof}

\begin{lemma}\label{lem: space increment bound}
    For any $t>0$ and $x,y
    \in \R$, we have $$\E[|u(t,x)-u(t,y)|^2]\lesssim t|x-y|,$$
    where the implicit constant depends only on $a$, $m$.
\end{lemma}

\begin{proof}
    % We have
    % \[u(t,x)-u(t,y)=\int_0^t\int_{-\infty}^{\infty}\mathcal{F}^{-1}[(e^{- ix(\cdot)}-e^{-iy(\cdot)})\hat{G}(s,\cdot)](z)W(dz,ds),\]
    We can write
    \begin{align*}
        \E[\abs{u(t,x)-u(t,y)}^2]
        =&\int_0^t \int_{-\infty}^\infty [G(s, x-z) - G(s, y-z)]^2 dzds\\
        \propto&\int_0^t\int_{-\infty}^{\infty}\abs{e^{-ix\xi}-e^{iy\xi}}^2\hat{G}(s,\xi)^2d\xi ds\\
        \leq& \int_0^t\int_{-\infty}^{\infty} \left(4 \wedge (|x-y|^2|\xi|^2) \right)\hat{G}(s,\xi)^2d\xi ds\\
        \lesssim& \int_0^t \int_0^{1/|x-y|} |x-y|^2|\xi|^2 \hat{G}(s,\xi)^2d\xi ds + \int_0^t\int_{1/|x-y|}^\infty |\hat{G}(s,\xi)|^2 d\xi ds.
    \end{align*}
    Since $\sin(x)/x$ is bounded, $\sinh(x)/x$ is locally bounded, and $\hat{G}(t,\xi) = t e^{-at/2}$ when $\xi^2 = a^2/4-m^2$,
    we have the bound $|\hat{G}(s,\xi)|^2\lesssim\frac{1}{\xi^2}$ for all $s \in [0,T]$ and $\xi>0$, which gives $$\int_0^t\int_{1/|x-y|}^\infty |\hat{G}(s,\xi)|^2 d\xi ds\lesssim \int_0^t\int_{1/|x-y|}^\infty \frac{1}{\xi^2}d\xi ds\leq t|x-y|.$$
    %\[
    %|e^{-ix\xi}-e^{-iy\xi}|\le 2 \wedge (|x-y||\xi|)
    %\]
    %\[
    %\int_{0}^{1/|x-y|} |x-y|^2 |\xi|^2 |\hat{G}|^2 d\xi
    %+ \int_{1/|x-y|}^\infty |\hat{G}|^2 d\xi
    %\]
    %First term: divide into 2 parts; in both regimes $\xi^2 |\hat{G}|^2$ is bounded?\\
    %Second term: $\int_{1/|x-y|}^\infty \xi^{-2} d\xi = |x-y|$
    If $|x-y|>1$, then $|\xi|<1/|x-y|$ implies $|\xi|<1$. This together with $|\hat{G}(s,\xi)| \lesssim 1$ for $|\xi|<1$ gives
    \begin{align*}
        \int_0^t \int_0^{1/|x-y|} |x-y|^2|\xi|^2 \hat{G}(s,\xi)^2d\xi ds
        \lesssim |x-y|^2 \int_0^t \int_0^{1/|x-y|}  d\xi ds
        = t |x-y|.
    \end{align*}
    Finally, if $|x-y|\le 1$, we use $\int_0^{1/|x-y|}=\int_0^1+\int_1^{1/|x-y|}.$ 
    The first estimate and $|\xi|<1$ then gives $$\int_0^t \int_0^{1} |x-y|^2|\xi|^2 \hat{G}(s,\xi)^2d\xi ds\lesssim t |x-y|^2.$$
    For $|\xi|>1$, we again use $|\hat{G}(s,\xi)|^2\lesssim \frac{1}{\xi^2}$, so $$\int_0^t \int_1^{1/|x-y|} |x-y|^2|\xi|^2 \hat{G}(s,\xi)^2d\xi ds\lesssim t(|x-y|^2-|x-y|)\lesssim t|x-y|.$$
    This finishes the proof of Lemma \ref{lem: space increment bound}.
\end{proof}

Combining Lemma \ref{lem: time increment bound} and Lemma \ref{lem: space increment bound}, and applying the Kolmogorov continuity theorem, we obtain:

\begin{proposition}\label{lem:u-u}
    For any $T>0$, there exists $C>0$ such that
    \begin{align}\label{E:u-u}
        \E[|u(t,x)-u(s,y)|^2] \le C(|t-s|+|x-y|)
    \end{align}
    uniformly for all $s,t \in [0,T]$ and $x,y \in \R$.
    Hence, the sample function $(t,x) \mapsto u(t,x)$ is a.s. continuous on $[0,\infty) \times \R$ and locally H\"older of order $1/2-$ in both $t$ and $x$.
\end{proposition}

\section{Reduction to the critical damping case}
\label{s:reduction}

We now explain how the proofs of the main theorems can be reduced to the critical damping case $m^2=a^2/4$.
This transformation is a standard PDE technique.
Note that this simplification immediately implies all of the main theorems if $a=0$ by \cite{Walsh86,CN88,LeeXiao22}, so for all future sections we will assume $a>0$ to avoid triviality. 
% The main tool for the reduction is the following version of Girsanov's theorem, whose proof can be found in \cite{Allouba}.
% 
% \begin{lemma}[Girsanov's theorem]
%     Fix $T>0$ and let $\{Z(t,x)\}_{(t,x) \in [0,T]\times \R}$ be a predictable process that satisfies $\E[\operatorname{exp}{(\frac12 \int_0^T \int_\R |Z(t,x)|^2 dx dt)}]<\infty$.
%     Then $\tilde{W}(dt, dx) := W(dt, dx) + Z(t,x) dt dx$ is a space-time white noise on $(\Omega, \mathscr{F}_T, Q)$, where
%     \[
%         \frac{dQ}{d\P} = \exp\left( - \int_0^T \int_\R Z(t,x) W(dt, dx) - \int_0^T \int_\R |Z(t,x)|^2 dx dt \right).
%     \]
% \end{lemma}
% 
% We now explain why it is enough to prove our main results for the case $m^2=a^2/4$.
First, re-write equation \eqref{eq:DSKG} as:
\begin{align}\label{eq:DSKG+b}
    \begin{cases}
        (\DAlambert + a\partial_t + a^2/4) u(t,x) = \dot{W} (t,x) + b(u(t,x)), \quad t>0,x\in \R,\\
        u(0,x) = \partial_t u(0,x) = 0,
    \end{cases}
\end{align}
where $b:\R \to \R$ is the function defined by
\begin{align*}
    b(x) = (a^2/4-m^2) x.
\end{align*}
Since $b$ is globally Lipschitz and grows linearly, standard existence and uniqueness theory implies that the solution $u$ to \eqref{eq:DSKG} is also the unique mild solution to \eqref{eq:DSKG+b} which solves the integral equation
\begin{align}\label{eq:DSKG+b:mild}
    u(t,x) = \int_0^t \int_\R \Gamma(t-s, x-y) W(ds, dy) + \int_0^t \int_\R \Gamma(t-s,x-y) b(u(s,y)) dy ds, \quad t>0,x\in \R,
\end{align}
where $\Gamma$ is the fundamental solution to $\DAlambert + a \partial_t + a^2/4$.
% For each $N>0$, define $b_N: \R \to \R$ by truncation:
% \[
%     b_N(x) = b(x) \wedge N.
% \]
% Let $u_N$ be the solution to
% \begin{align*}
%     \begin{cases}
%         (\DAlambert + a\partial_t + a^2/4) u_N(t,x) = \dot{W} (t,x) + b_N(u_N(t,x)), \quad t>0,x\in \R,\\
%         u_N(0,x) = \partial_t u_N(0,x) = 0.
%     \end{cases}
% \end{align*}
% Define
% \[
%     \tau_N = 
% \]
Notice that in this case, we have by \eqref{Fourier G} \begin{equation}\label{G: critical}
     \hat{\Gamma}(t,\xi) = \frac{e^{-at/2}\sin(t|\xi|)}{|\xi|}
     \quad \text{and hence} \quad
     \Gamma(t,x) = \frac{e^{-at/2}}{2} \1_{\{|x|< t\}}.
    % \Gamma(t,x)=\frac{e^{-\frac{at}{2}}}{2}\1_{\{|x|< t\}}\implies
    % \Gamma(t,x)=\frac{e^{-\frac{at}{2}}}{2}\1_{\{|x|< t\}}.
\end{equation}
As mentioned in the preliminaries, this special case is referred to as critical damping and significantly simplifies our problem. To show that the simplification is valid, we write \eqref{eq:DSKG+b:mild} as
\begin{align}\label{E:u=u_C+u_L}
    &u(t,x) = u_C(t,x) + u_L(t,x),\text{ where }\\
    &u_C(t,x) = \int_0^t \int_\R \Gamma(t-s, x-y) W(ds, dy),\\
    &u_L(t,x) = (a^2/4-m^2) \int_0^t \int_\R \Gamma(t-s,x-y) u(s,y) dy ds.
\end{align}
In particular, $u_C$ is the mild solution to \eqref{eq:DSKG} with $m^2 = a^2/4$.
In the following lemma, we show that $u_L$ is a.s. locally Lipschitz (along characteristic direction).

\begin{lemma}\label{lem:u_L}
    For any fixed compact set $F \subset [0,\infty) \times \R$, there exists $C>0$ such that
    \begin{align}\label{E:u_L:Lip}
        \sup_{h \in [0,1]}\sup_{(t,x)\in F} \frac{|u_L(t+h,x+h)-u_L(t,x)|}{h} \le C \sup_{(s,y)\in \mathscr{C}(F)}|u(s,y)| < +\infty \quad \text{a.s.}
    \end{align}
    where $\mathscr{C}(F) = \bigcup_{h \in [0,1]}\bigcup_{(t,x)\in F}\{(s,y) \in [0,\infty) \times \R : |x+h-y| \le t+h-s\}$.
\end{lemma}

\begin{proof}
    We first apply absolute value to the integrand to see that
    \begin{align*}
        &|u_L(t+h,x+h)-u_L(t,x)|\\
        &\le |\tfrac{a^2}{4}-m^2| \int_{\R_+} \int_\R \left|\Gamma(t+h-s,x+h-y) \1_{\{|x+h-y| \le t+h-s\}} - \Gamma(t-s,x-y) \1_{\{|x-y| \le t-s\}}\right| |u(s,y)| dy ds\\
        & \le |\tfrac{a^2}{4}-m^2| \sup_{(s,y)\in \mathscr{C}(F)} |u(s,y)|\, (I_1+I_2+I_3),
    \end{align*}
    where
    \begin{align*}
        &I_1 = \int_0^t \int_{x-(t-s)}^{x+(t-s)} |\Gamma(t+h-s,x+h-y)  - \Gamma(t-s,x-y)| dyds,\\
        &I_2 = \int_0^t \int_{x+(t-s)}^{x+h+(t+h-s)} \Gamma(t+h-s,x+h-y) dyds,\\
        &I_3 = \int_t^{t+h} \int_{x+h-(t+h-s)}^{x+h+(t+h-s)} \Gamma(t+h-s,x+h-y) dyds.
    \end{align*}
    Then, we may use \eqref{G: critical} to estimate these three terms easily:
    \begin{align*}
        I_1 &= \frac12 \int_0^t \int_{x-(t-s)}^{x+(t-s)} (e^{-a(t-s)/2}-e^{-a(t+h-s)/2}) dyds
        = (1-e^{-ah/2}) \int_0^t (t-s) e^{-a(t-s)/2} ds \lesssim h,
    \end{align*}
    \begin{align*}
        I_2 & = \frac12 \int_0^t \int_{x+(t-s)}^{x+(t-s)+2h} e^{-a(t+h-s)/2} dyds
        \le h \int_0^t e^{-a(t-s)/2} ds \le h,
    \end{align*}
    \begin{align*}
        I_3 = \frac12 \int_t^{t+h} \int_{x-(t-s)}^{x+(t-s)+2h} e^{-a(t+h-s)} dyds \le h^2.
    \end{align*}
    Finally, we finish the proof using the property that $u$ is a.s. bounded on the compact set $\mathscr{C}(F)$, due to the a.s. continuity of $u$ (see Proposition \ref{lem:u-u}).
\end{proof}

\subsection{Reduction of Theorem \ref{th:lil}}

Fix $t_0 \ge 0$ and $x_0 \in \R$.
Assume that Theorem \ref{th:lil} holds for $u_C$, i.e., there exists a constant $0<K_1<\infty$ such that $\P(\Omega_1) = 1$, where
\begin{align*}
    \Omega_1 = \left\{\limsup_{h \to 0^+} \frac{|u_C(t_0+h,x_0+h)-u_C(t_0,x_0)|}{\sqrt{h \log\log(1/h)}} = K_1 \right\}.
\end{align*}
Thanks to Lemma \ref{lem:u_L}, we may take $F= \{(t_0,x_0)\}$ and find an event $\Omega_2$ with $\P(\Omega_2) = 1$ on which \eqref{E:u_L:Lip} holds.
This implies that for all $\omega\in\Omega_2$,
\begin{align*}
    \limsup_{h \to 0^+} \frac{|u_L(t_0,x_0+h)-u_L(t_0,x_0)|}{\sqrt{h \log\log(1/h)}} = 0.
\end{align*}
Hence, we may use \eqref{E:u=u_C+u_L} and the elementary inequality
\begin{align}\label{limsup:ineq}
    \limsup |f(x)| - \limsup |g(x)| \le \limsup |f(x)+g(x)| \le \limsup |f(x)| + \limsup |g(x)|
\end{align}
(provided $\limsup |g(x)|<+\infty$)
to deduce that \eqref{E:lil} holds on the event $\Omega_1 \cap \Omega_2$ which has probability 1.

\subsection{Reduction of Theorem \ref{th:mc}}

Fix $w_0 \ge 0$ and $0<a_1<a_2$.
Assume that Theorem \ref{th:mc} holds for $u_C$, i.e., we can find constants $0<C_1<C_2<\infty$ such that for every sub-interval $J$ of $I$, there exists a constant $K_J \in [C_1,C_2]$ such that $\P(\Omega_J) = 1$, where
\begin{align*}
    \Omega_J = \left\{ \limsup_{h\to0^+} \sup_{(t,x) \in J}\frac{|u_C(t+h,x+h)-u_C(t,x)|}{\sqrt{h\log(1/h)}} = K_J \right\}.
\end{align*}
Taking $F = I$ in Lemma \ref{lem:u_L}, we can find an event $\Omega_3$ with $\P(\Omega_3) = 1$ on which \eqref{E:u_L:Lip} holds.
It follows that for all $\omega \in \Omega_3$,
\begin{align*}
    \limsup_{h\to0^+}\sup_{(t,x)\in I} \frac{|u_L(t+h,x+h)-u_L(t,x)|}{\sqrt{h\log(1/h)}} = 0.
\end{align*}
Then, apply \eqref{E:u=u_C+u_L} and \eqref{limsup:ineq} to see that \eqref{E:mc} holds on the event $\Omega_J \cap \Omega_3$ which has probability 1.

\subsection{Reduction of Theorem \ref{th:sim:lil}}

In order to simplify notations, write
\begin{align*}
    \ell(w,z) &= \limsup_{h\to0^+} \frac{|u\big(\frac{w+z}{\sqrt{2}}+h, \frac{-w+z}{\sqrt{2}}+h\big) - u\big(\frac{w+z}{\sqrt{2}}, \frac{-w+z}{\sqrt{2}}\big)|}{\sqrt{h\log\log(1/h)}},\\
    \ell_i(w,z) &= \limsup_{h\to0^+} \frac{|u_i\big(\frac{w+z}{\sqrt{2}}+h, \frac{-w+z}{\sqrt{2}}+h\big) - u_i\big(\frac{w+z}{\sqrt{2}}, \frac{-w+z}{\sqrt{2}}\big)|}{\sqrt{h\log\log(1/h)}}, \quad i = C, L.
\end{align*}
Fix $N>0$, $0<a_1<a_2$ and assume that Theorem \ref{th:sim:lil} holds for $u_C$.
This implies that there exists a constant $0<K_2<\infty$ such that $\P(A_{z_0}) = 1$ for all $z_0 \in [a_1,a_2]$, where
\begin{align*}
    A_{z_0} = \left\{ \ell_C(w,z_0) \le K_2 \text{ for all $w \in [0,N]$} \right\}.
\end{align*}
An application of Lemma \ref{lem:u_L} with
\[
    F = \left\{ (t,x) = (\tfrac{w+z}{\sqrt{2}}, \tfrac{-w+z}{\sqrt{2}}) : w \in [0,N], z \in [a_1,a_2] \right\}
\]
shows that there is an event $B$ with $\P(B)=1$ on which \eqref{E:u_L:Lip} holds.
It is clear that for any $z_0 \in [a_1,a_2]$,
\begin{align*}
    B \subset \left\{\ell_L(w,z_0) = 0  \text{ for all $w \in [0,N]$}\right\}.
\end{align*}
It follows from \eqref{E:u=u_C+u_L} and \eqref{limsup:ineq} that on the event $A_{z_0} \cap B$, which has probability 1, we have
\[
    \ell(w,z_0) \le \ell_C(w,z_0) + \ell_L(w,z_0) \le K_2 \quad \text{for all $w \in [0,N]$.}
\]
% Finally, we may take $\bigcap_{N\in \N_+} (A_N\cap B_N)$ as the desired event for which \eqref{E:sim:lil} holds.

\subsection{Reduction of Theorem \ref{th:sing}}

Fix $w_0 \ge 0$, $0 < z_0 < z_0'$ and let $t_0=w_0/\sqrt{2}$.
Assume that part (i) Theorem \ref{th:sing} holds for $u_C$.
Then we can find an $\mathscr{F}_{t_0}$-measurable random variable $Z=Z(\omega)$ such that $\P(\Omega_4)=1$, where $\Omega_4$ denotes the event
\[
    \Omega_4 = \{Z \in [z_0, z_0'] \text{ and } \ell_C(w_0,Z) = + \infty\}.
\]
Taking
\[
    F = \left\{ (t,x) = (\tfrac{w+z}{\sqrt{2}}, \tfrac{-w+z}{\sqrt{2}}) : w =w_0, z \in [z_0,z_0'] \right\}
\]
in Lemma \ref{lem:u_L}, we can find another event $\Omega_5$ with $\P(\Omega_5)=1$ on which \eqref{E:u_L:Lip} holds.
In particular, this implies that
\[
    \Omega_5 \subset \{ \ell_L(w_0,z) = 0 \text{ for all $z \in [z_0,z_0']$}\}.
\]
Therefore, $\Omega_4 \cap \Omega_5$ is an event with probability 1 on which
\[
    \ell(w,Z) \ge \ell_C(w,Z) - \ell_L(w,Z) = +\infty,
\]
thanks to \eqref{E:u=u_C+u_L} and \eqref{limsup:ineq}. This proves that part (i) of Theorem \ref{th:sing} holds for $u$.

Let $Z$ be any $\mathscr{F}_{t_0}$-measurable random variable with $\P(\Omega_6)=1$, where
\[
    \Omega_6 = \{Z \in [z_0, z_0'] \text{ and } \ell_C(w_0,Z) = + \infty\},
\]
and assume that part (ii) of Theorem \ref{th:sing} holds for $u_C$. In particular, this assumption implies that $\P(E_N)=1$ for all $N>w_0$, where
\[
    E_N = \{ \ell_C(w,Z) = +\infty \text{ for all $w \in [w_0,N]$} \}.
\]
For each $N >w_0$, we may apply Lemma \ref{lem:u_L} with
\[
    F = \left\{ (t,x) = (\tfrac{w+z}{\sqrt{2}}, \tfrac{-w+z}{\sqrt{2}}) : w \in [w_0, N], z \in [z_0,z_0'] \right\}
\]
to see there is an event $G_N$ with $\P(G_N)=1$ on which \eqref{E:u_L:Lip} holds.
In particular,
\[
    G_N \subset \{\ell_L(w,Z) = 0 \text{ for all $w \in [w_0,N]$}\}.
\]
It follows that on $E_N \cap G_N$,
\[
    \ell(w,Z) \ge \ell_C(w,Z) - \ell_L(w,Z) = +\infty \quad \text{for all $w\in [w_0,N]$.}
\]
Hence, $\bigcap_{N \in \N, N>w_0}(E_N \cap G_N)$ is an event with probability 1 on which \eqref{E:propagation} holds.
Therefore, part (ii) of Theorem \ref{th:sing} holds for $u$.

\section{Proof of Theorems \ref{th:lil} and \ref{th:mc}}
\label{s:4}

Thanks to the reduction in Section \ref{s:reduction}, we assume in the remainder of the paper that $m^2=a^2/2$.
Denote by $\Tilde{F}[f](\tau,\xi)$ the space-time fourier transform of an arbitrary $f(t,x)$ and $\Tilde{\Gamma}(\tau,\xi):=\Tilde{F}[\Gamma(\cdot,\circ)](\tau,\xi)$.

We will apply the results of \cite{LX23} to prove Theorems \ref{th:lil} and \ref{th:mc}.
To this end, we first establish some technical lemmas in order to verify the assumptions in \cite{LX23}.
Let $\tilde{W}$ be a complex-valued space-time white noise, i.e., $\tilde{W} = \tilde{W}_1 + i \tilde{W}_2$, where $\tilde{W}_1, \tilde{W}_2$ are independent space-time white noises.

\begin{lemma}\label{lemma: Polarity 8.3}
    Fix $T>0$. Define the Gaussian random field $v(t,x)$ by \begin{equation}\label{eq: v}
       v(t,x)=\Re \int_\R\int_\R \Tilde{F}[\Gamma(t-\cdot,x-\circ)\1_{[0,t]}(\cdot)](\tau,\xi)\Tilde{W}(d\tau,d\xi),
    \end{equation}
    and the truncated random field $v(A,t,x)$ for a Borel set $A \subset [0,\infty)$ by 
    \begin{equation}\label{eq: vab}
       v(A,t,x)=\Re \iint_{|\tau|\vee |\xi|\in A} \Tilde{F}[\Gamma(t-\cdot,x-\circ)\1_{[0,t]}(\cdot)](\tau,\xi)\Tilde{W}(d\tau,d\xi).
    \end{equation}
    Then, $\{ v(t,x) : (t,x) \in \R_+ \times \R \}$ has the same law as $\{u(t,x): (t,x) \in \R_+\times \R\}$, and $v(A,t,x)$ is a centered Gaussian random field such that $v(A,\cdot)$ and $v(B,\cdot)$ are independent whenever $A$ and $B$ are disjoint.
    Moreover, there exists $a_0 > 0$ such that for all $a_0 \le a < b \le \infty$, $s,t \in [0,T]$ and $x,y \in \R$, we have 
    \begin{align*}
        &\E[\{v([a,b),t,x)-v([a,b),s,y)-v(t,x)+v(s,y)\}^2]^{1/2}\\
        &\lesssim a(|t-s|+|x-y|)+b^{-1}
    \end{align*}
    and
    \begin{align*}
        \E\left[ (v([0,a_0],t,x) - v([0,a_0],s,y) \right]^{1/2} \lesssim |t-s| + |x-y|.
    \end{align*}
\end{lemma}

The random field $v$ in \eqref{eq: v} is called the harmonizable representation of $u$ \cite{LX23}.
We will need some elementary lemmas for the proof of Lemma \ref{lemma: Polarity 8.3}.
\begin{lemma}\label{lemma: crit damp FS and grad bd}
    Define $F(t,x,\tau,\xi):=\tilde{F}[\Gamma(t-\cdot,x-\circ)\1_{[0,t]}(\cdot)](\tau,\xi)$. For any $x,x'\in \R$, $0<t,t'<T$, the following holds.
    \begin{enumerate}[label=(\alph*)]
        \item $|F(t,x,\tau,\xi)|\lesssim \frac{1}{1+|\xi|}\left(\frac{1}{1+\frac{1}{2}|\tau+|\xi||}+\frac{1}{1+\frac{1}{2}|\tau-|\xi||}\right)$.
        \item Given another $0<t'<T$ and $x'\in \R$, $$|F(t,x,\tau,\xi)-F(t',x',\tau,\xi)|\lesssim (|t-t'|+|x-x'|)\left[ \frac{1}{1+\frac{1}{2}|\tau+|\xi||}+\frac{1}{1+\frac{1}{2}|\tau-|\xi||}\right].$$
    \end{enumerate}
\end{lemma}
\begin{proof}
    This lemma is the analogue of \cite[Lemma 9.4]{DMX_Polarity} (with the results listed in reverse for convenience), and the proof will go similarly.
    By an elementary calculation, we have \begin{align}\label{eq: F}
        F(t,x,\tau,
        \xi)=& \frac{e^{-it\tau+ix\xi}}{2i|\xi|}[f(t,z_+)-f(t,z_-)],\text{ where }\nonumber\\
        f(t,z):=&\frac{e^{tz}-1}{z},\quad z_{\pm}:=i\tau-\frac{a}{2}\pm i\xi.
    \end{align}
    For the proof of part (a),  \eqref{eq: F} gives $$|F(t,x,\tau,\xi)|=\frac{|\phi_t(y_+)-\phi_t(y_-)|}{2|y_+-y_-|},\text{ where }\phi_t(y)=\frac{e^{it(y+i\frac{a}{2})-1}}{i(y+i\frac{a}{2})}.$$
    %{\color{red}Clearly, the substitution $\tilde{y}=y+i\frac{a}{2}$ gives us $\phi_t(y)=\tilde{\phi}_t(\tilde{y})=\frac{e^{it\tilde{y}-1}}{\tilde{y}}$ and $\phi_t'(y)=\tilde{\phi}_t'(y)$ by the chain rule (note that since $y$ is purely real, we need not worry about $\tilde{y}=0$ if $a\neq 0$). Now, $$\tilde{\phi}_t'(\tilde{y})=\frac{tye^{ty}+1-e^{ty}}{y^2}$$}
    Observe that since $y_+-y_-=(y_++i\frac{a}{2})-(y_-+i\frac{a}{2})$, we have by the substitution $\tilde{y}_\pm=y_\pm+i\frac{a}{2}$ $$|F(t,x,\tau,\xi)|=\frac{|\tilde{\phi}_t(\tilde{y}_+)-\tilde{\phi}_t(\tilde{y}_-)|}{|\tilde{y}_+-\tilde{y}_-|},\text{ where }\tilde{\phi}_t(\tilde{y})=\frac{e^{it\tilde{y}}-1}{\tilde{y}}.$$
    Note here that $\tilde{\phi}_t$ is a map $\C\supset\R+i\frac{a}{2}\to \C$, so we need not worry about $\tilde{y}=0$. However, it is easy to show using elementary arguments that $$\sup_{(t,\tilde{y})\in [0,T]\times \R+i\frac{a}{2}}\max(\tilde{\phi}_t(\tilde{y}),\tilde{\phi}_t'(\tilde{y}))\lesssim \frac{1}{1+|\tilde{y}|}\lesssim \frac{1}{1+|y|}.$$
    Thus, the elementary arguments involving the mean value theorem in the proof of \cite[Lemma 9.4(b)]{DMX_Polarity}) give us $$\frac{|\tilde{\phi}(\tilde{y}_+)-\tilde{\phi}(\tilde{y}_-)|}{2|y_+-y_-|}\lesssim\frac{1}{1+2|y_+-y_-|}\left[\frac{1}{1+|y_+|}+\frac{1}{1+|y_-|}\right],$$ which completes the proof of part (a).
    %The rest of the proof of part (a) reduces to the same calculations used to prove the analogous statement in [Dalang, Mueller, Xiao'17].
    For part (b), we write \begin{align*}
        |F(t,x,\tau,\xi)-F(t',x',\tau,\xi)|=&|F(t,x,\tau,\xi)-F(t',x,\tau,\xi)+F(t',x,\tau,\xi)-F(t',x',\tau,\xi)|\\
        \leq& |F(t,x,\tau,\xi)-F(t',x,\tau,\xi)|+|F(t',x,\tau,\xi)-F(t',x',\tau,\xi)|
    \end{align*}
    from which we will apply the mean value theorem in the first and second variables of $F$, which will conclude the proof. For the space increment, we compute $$|\partial_x F|=|i\xi F|\leq \left[ \frac{1}{1+\frac{1}{2}|\tau+|\xi||}+\frac{1}{1+\frac{1}{2}|\tau-|\xi||}\right]$$
    Where the inequality is an application of part (a). Then the mean value theorem gives us \begin{equation}\label{est: F space inc}
        |F(t',x,\tau,\xi)-F(t',x',\tau,\xi)|\leq |x-x'|\left[ \frac{1}{1+\frac{1}{2}|\tau+|\xi||}+\frac{1}{1+\frac{1}{2}|\tau-|\xi||}\right].
    \end{equation}
    For the time increment, we use $\partial_tf(t,z)=e^{tz}$ to compute \begin{align*}
        \partial_t F&=-i\tau F+\frac{e^{-it\tau+ix\xi}}{2i|\xi|}[\partial_tf(t,z_+)-\partial_t f(t,z_-)]
        = -i\tau F+\frac{e^{-it\tau+ix\xi}}{2i|\xi|}[e^{tz_+}-e^{tz_-}]\\
        &=\frac{e^{-it\tau+ix\xi}}{2i|\xi|}\left[-i\tau(f(t,z_+)-f(t,z_-))+e^{tz_+}-e^{tz_-}\right]\\
        &=\frac{e^{-it\tau+ix\xi}}{2i|\xi|}[-i\tau(f(t,z_+)-f(t,z_-))+z_+f(t,z_+)-1-(z_-f(t,z_+)-1)]\\
        &=\frac{e^{-it\tau+ix\xi}}{2i|\xi|}[(z_+-i\tau)f(t,z_+)-(z_--i\tau)f(t,z_-)]\\
        &=\frac{e^{-it\tau+ix\xi}}{2i|\xi|}[(z_+-i\tau)f(t,z_+)-(z_--i\tau)f(t,z_+)+(z_--i\tau)f(t,z_+)-(z_--i\tau)f(t,z_-)]
    \end{align*}
    The fact $|\xi|=|z_+-z_-|$ gives us \begin{align*}
        |\partial_tF| &\leq \frac{|z_+-z_-||f(t,z_+)|}{2|z_+-z_-|}+|z_--i\tau|\frac{|f(t,z_+)-f(t,z_-)|}{2|z_+-z_-|}\\
        &\leq \frac{1}{2}|f(t,z_+)|+C\frac{|f(t,z_+)-f(t,z_-)|}{2|z_+-z_-|}\\
        &= \frac{1}{2}\frac{|(1-e^{-\frac{at}{2}})+e^{i(\tau+|\xi|)}|}{|i(\tau+|\xi|)+\frac{a}{2}|}+C\frac{|f(t,z_+)-f(t,z_-)|}{2|z_+-z_-|}\\
        &\lesssim \frac{1}{1+\frac{1}{2}|\tau+|\xi||}+\frac{1}{1+\frac{1}{2}|\tau-|\xi||}.
    \end{align*}
    For the last inequality above, the second term is again an application of part (a), and the argument for the first term is elementary. This gives us $$|F(t,x,\tau,\xi)-F(t',x,\tau,\xi)|\lesssim |t-t'|\left[ \frac{1}{1+\frac{1}{2}|\tau+|\xi||}+\frac{1}{1+\frac{1}{2}|\tau-|\xi||}\right],$$
    which finishes the proof.
\end{proof}

\begin{lemma}\label{lemma: polarity 9.5}
    \begin{enumerate}[label=(\alph*)]
        \item For any $a,\alpha>0$, $$\iint_{\max (|\theta|,|\zeta|)\leq 2a^{\frac{1}{\alpha}}}\frac{1}{1+\frac{1}{4}(\theta-\zeta)^2}+\frac{1}{1+\frac{1}{4}(\theta+\zeta)^2}d\theta d\zeta\lesssim a^{\frac{1}{\alpha}}.$$
        \item For sufficiently large $b$, $$\iint_{\max (|\theta|,|\zeta|)>b^{\frac{1}{\alpha}},r>0}\left[\frac{1}{1+\frac{1}{4}(\theta-\zeta)^2}+\frac{1}{1+\frac{1}{4}(\theta+\zeta)^2}\right]\frac{1}{1+\zeta^2}d\theta d\zeta\lesssim b^{-2}.$$
    \end{enumerate}
\end{lemma}
\begin{proof}
    This lemma is \cite[Lemma 9.5]{DMX_Polarity} with $\beta=1$.
\end{proof}

\begin{proof}[Proof of Lemma \ref{lemma: Polarity 8.3}]
    This lemma is an analogue of \cite[Lemma 9.3]{DMX_Polarity}, and the original proof can be re-written here with usage of Lemmas \ref{lemma: crit damp FS and grad bd} and \ref{lemma: polarity 9.5} (with $\alpha=1/2$) in the appropriate places with no changes in the current setting.
\end{proof}

\begin{lemma}\label{lem:SLND}
    Fix $w_0\ge 0$ and $0<a_1<a_2$, and consider the line segment $I$ defined by \eqref{line:I}.
    Then, there exists $C>0$ such that 
    \begin{align}\label{E:SLND}
        \Var(u(t,x)\mid u(t_1,x_1),\dots, u(t_n,x_n)) \ge C \min_{1\le i \le n}(|t-t_i|+|x-x_i|)
    \end{align}
    uniformly for all $n \in \N_+$ and $(t,x),(t_1,x_1),\dots,(t_n,x_n) \in I$ with $\max_{1 \le i \le n} t_i \le t$.
    In particular,
    \begin{align}\label{E:u-u:LB}
        \Var(u(t,x)-u(s,y)) \ge C(|t-s|+|x-y|)
    \end{align}
    uniformly for all $(t,x),(s,y) \in I$.
\end{lemma}

\begin{proof}
    Recall the basic property that
    \begin{align*}
        \Var(X\mid X_1,\dots, X_n) = \inf_{c_1,\dots,c_n \in \R} \Var \left( X - \sum_{i=1}^n c_i X_i \right)
    \end{align*}
    for any centered Gaussian vector $(X,X_1,\dots, X_n)$.
    Also, note that $|t-s|+|x-y| = 2|t-s|$ for any pair of points $(t,x)$, $(s,y)$ on the line segment $I$.
    Hence, it suffices to show the existence of $C>0$ such that
    \begin{align*}
        \Var\left( u(t,x) - \sum_{i=1}^n c_i u(t_i,x_i) \right) 
        \ge C (|t-t_n|+|x-x_n|)
    \end{align*}
    uniformly for all $n \in \N_+$, $c_1,\dots, c_n \in \R$, and $(t,x),(t_1,x_1),\dots,(t_n,x_n) \in I$ with $t_1\le \cdots \le t_n \le t$.
    To this end, we define
    \begin{align*}
        &u_0(t,x,t_n,x_n) = \int_0^{t_n} \int_{x_n+t_n-s}^{x+t-s} \Gamma(t-s,x-y) W(ds, dy), \\
        &u_1(t,x,t_n,x_n) = u(t,x)-u_0(t,x,t_n,x_n).
    \end{align*}
    Note that $u_0(t,x,t_n,x_n)$ is a Wiener integral over $A:=\{(s,y): 0 < s < t, x_n+t_n-s < t < x+t-s\}$, while
    $u_1(t,x,t_n,x_n), u(t_1,x_1), \dots, u(t_n,x_n)$ are Wiener integrals over regions that are all contained in $B:=\{ (s,y) : 0<s<t_n, x_n-t_n+s < y < x_n+t_n-s \} \cup \{ (s,y) : t_n<s<t, x-t+s<y<x+t-s \}$ which is disjoint from $A$.
    Hence, we may use independence property of the white noise $\dot{W}$ to see that
    \begin{align*}
        &\Var\left( u(t,x) - \sum_{i=1}^n c_i u(t_i,x_i) \right) \\
        &= \Var\left( u_1(t,x,t_n,x_n) - \sum_{i=1}^n c_i u(t_i,x_i)\right) + \Var(u_0(t,x,t_n,x_n))
        \ge \Var(u_0(t,x,t_n,x_n)).
    \end{align*}
    By It\^o-Walsh isometry,
    \begin{align*}
        \Var(u_0(t,x,t_n,x_n))
        = \frac14 \int_0^{t_n} \int_{x_n+t_n-s}^{x+t-s} e^{-a(t-s)} dy ds \ge C (t+x - t_n-x_n),
    \end{align*}
    since $w_0\le t_n\le t$ and $x_n \le x$. This completes the proof.
\end{proof}

% We now state some results from \cite{LX23} from which the lemmas in this section will immediately imply Theorems \ref{th:lil} and \ref{th:mc}.

% \vspace{0.5cm}
% {\color{red}cited theorems go here.}
% \vspace{0.5cm}

\begin{proof}[Proof of Theorem \ref{th:lil}]
    Fix $w_0\ge 0$ and $0<a_1<a_2$.
    Recall \eqref{w,z} and consider the process $\tilde{u}$ defined by
    \begin{align}\label{E:tilde:u}
        \tilde{u}(z) = u\left( \tfrac{w_0+z}{\sqrt{2}}, \tfrac{-w_0+z}{\sqrt{2}} \right), \quad z \in [a_1,a_2].
    \end{align}
    Thanks to Lemma \ref{lemma: Polarity 8.3} and \eqref{E:u-u:LB} above, Assumptions 2.1 and 2.3 in \cite{LX23} are satisfied for $\tilde{u}$.
    Therefore, we may apply Theorem 5.2 of \cite{LX23} to obtain the desired LIL directly.
\end{proof}

%\section{Proof of Theorem \ref{th:mc}}

% \begin{lemma}\label{lem:elem:LB}
%     Fix $0<t_0<T_0$ and $h_0>0$. Then, there exists $C>0$ such that
%     \begin{align*}
%         \int_0^{t_0} \int_{x+t-s}^{x+t-s+2h} e^{-a(t+h-s)} dy \, ds
%         \ge C h
%     \end{align*}
%     uniformly for all $t \in [t_0, T_0]$, $x \in \R$ and $h \in [0,h_0]$.
% \end{lemma}

% \begin{proof}
%     The inequality can be shown by elementary calculations.
% \end{proof}

\begin{proof}[Proof of Theorem \ref{th:mc}]
    Again, Assumptions 2.1 and 2.3 of \cite{LX23} are satisfied for $\tilde{u}$ defined in \eqref{E:tilde:u}. 
    Here, we do not have the exact form of strong local nondeterminism (SLND) in Assumption 2.2 of \cite{LX23}, but \eqref{E:SLND} in Lemma \ref{lem:SLND} above shows that $\tilde{u}$ still satisfies one-sided SLND in the sense that
    \begin{align}\label{E:1-SLND}
        \Var(\tilde{u}(z)\mid \tilde{u}(z_1),\dots,\tilde{u}(z_n)) \ge C \min_{1\le i \le n}(z-z_i)
    \end{align}
    uniformly for all $n \in \N_+$ and $z,z_1,\dots,z_n \in [a_1,a_2]$ with $\max_{1\le i \le n} z_i \le z$.
    It is enough to use \eqref{E:1-SLND} in place of Assumption 2.2 in \cite{LX23} to repeat the proof of Theorem 6.1 of \cite{LX23} (where $Q=2$) and obtain the desired exact uniform modulus of continuity result \eqref{E:mc} with $\sqrt{4c_2}\le K_J \le \sqrt{2}c_1$, where $c_1$ is the constant in \eqref{E:u-u} and $c_2$ is the constant in \eqref{E:1-SLND}.
    Both constants depend only on $(w_0,a_1,a_2)$ but not on particular sub-intervals $J$ in $I$.
    This completes the proof.
\end{proof}

\section{Proof of Theorem \ref{th:sim:lil}}
\label{s:5}

We will need the following lemma, which is the analogue of \cite[Lemma 2.4]{LeeXiao22}.

\begin{lemma}\label{lem:rec_incr}
    Fix $T>0$. Then
    \begin{align*}
        \E\left[\left(u(t,x)-u(t-\tfrac{\varepsilon_1}{\sqrt{2}},x+\tfrac{\varepsilon_1}{\sqrt{2}})-u(t-\tfrac{\varepsilon_2}{\sqrt{2}}, x-\tfrac{\varepsilon_2}{\sqrt{2}})+u(t-\tfrac{\varepsilon_1+\varepsilon_2}{\sqrt{2}}, x+\tfrac{\varepsilon_1-\varepsilon_2}{\sqrt{2}})\right)^2\right] \lesssim \varepsilon_1 \varepsilon_2
    \end{align*}
    uniformly for all $t\in [0,T]$, $x\in \R$ and  $0<\varepsilon_1,\varepsilon_2<t$.
\end{lemma}

\begin{proof}
    By It\^o-Walsh isometry, the left hand side of the desired inequality is \begin{align*}
        &\int_0^t \int_{\R} \big(\Gamma(t-s,x-y)-\Gamma(t-\tfrac{\varepsilon_1}{\sqrt{2}}-s,x+\tfrac{\varepsilon_1}{\sqrt{2}}-y)-\Gamma(t-\tfrac{\varepsilon_2}{\sqrt{2}}-s, x-\tfrac{\varepsilon_2}{\sqrt{2}}-y)\\
        &+\Gamma(t-\tfrac{\varepsilon_1+\varepsilon_2}{\sqrt{2}}-s, x+\tfrac{\varepsilon_1-\varepsilon_2}{\sqrt{2}}-y)\big)^2dyds.
    \end{align*}
    For this proof, we will employ the simpler notations 
    \begin{align*}
        \Gamma:=&\Gamma(t-s,x-y),\quad \Gamma_1:=\Gamma(t-\tfrac{\varepsilon_1}{\sqrt{2}}-s,x+\tfrac{\varepsilon_1}{\sqrt{2}}-y),\\
        \Gamma_2:=&\Gamma(t-\frac{\varepsilon_2}{\sqrt{2}}-s, x-\tfrac{\varepsilon_2}{\sqrt{2}}-y),\text{ and }\Gamma_{12}:=\Gamma(t-\frac{\varepsilon_1+\varepsilon_2}{\sqrt{2}}-s, x+\tfrac{\varepsilon_1-\varepsilon_2}{\sqrt{2}}-y).
    \end{align*}
    In the following, we define some subsets of the open half plane $\set{y\in \R,s>0}$ given $x\in \R,0<\varepsilon_1,\varepsilon_2<t$ from the statement of the lemma.
    \begin{align*}
        &R:=\set{|x-y|<t-s},\quad R_1:=\left(R\setminus \Set{\left|x-\frac{\varepsilon_2}{\sqrt{2}}-y\right|<t-\frac{\varepsilon_2}{\sqrt{2}}-s}\right)\cap\Set{\left|x+\frac{\varepsilon_1}{\sqrt{2}}-y\right|<t-\frac{\varepsilon_1}{\sqrt{2}}-s}\\
        &R_2:=\left(R\setminus \Set{\left|x+\frac{\varepsilon_1}{\sqrt{2}}-y\right|<t-\frac{\varepsilon_1}{\sqrt{2}}-s}\right)\cap\Set{\left|x-\frac{\varepsilon_2}{\sqrt{2}}-y\right|<t-\frac{\varepsilon_2}{\sqrt{2}}-s},\\
        &R_{12}:=R\cap \Set{\left|x+\frac{\varepsilon_1}{\sqrt{2}}-y\right|<t-\frac{\varepsilon_1}{\sqrt{2}}-s} \cap \Set{\left|x-\frac{\varepsilon_2}{\sqrt{2}}-y\right|<t-\frac{\varepsilon_2}{\sqrt{2}}-s},\quad \Tilde{R}_{12}:=R\setminus (R_1\cup R_2\cup R_{12}).
    \end{align*}
    We now list some useful observations about the above sets.\begin{enumerate}
        \item $R$ is a disjoint union of $R_1,R_2,R_{12}$, and $\Tilde{R}_{12}$. It is the support of $\Gamma$ and contains the support of $\Gamma_1,\Gamma_2$, and $\Gamma_{12}$.
        \item $\Tilde{R}_{12}$ has area $\varepsilon_1\varepsilon_2$ and only $\Gamma$ is nonzero on it.
        \item $R_1$ has area smaller than $\varepsilon_2 T$ and only $G$ and $\Gamma_1$ are nonzero on it. Similarly, $R_2$ has area smaller than $\varepsilon_1 T$ and only $\Gamma$ and $\Gamma_2$ are nonzero on it.
        \item $R_{12}$ has area no larger than $\frac{1}{2}(t-\varepsilon_1-\varepsilon_2)^2<T^2$ and $\Gamma,\Gamma_1,\Gamma_2$ and $\Gamma_{12}$ are all nonzero on it.
    \end{enumerate}
    By the above observations, we can write the left hand side of our estimate as $\Tilde{I}_{12}+I_1+I_2+I_{12}$, where \begin{align*}
        \Tilde{I}_{12}:=&\iint_{\Tilde{R}_{12}} \Gamma^2 dyds,\quad I_1:=\iint_{R_1} (\Gamma-\Gamma_1)^2 dyds,\\
        I_2:=& \iint_{R_2} (\Gamma-\Gamma_2)^2 dyds,\text{ and }I_{12}:=\iint_{R_{12}} (\Gamma-\Gamma_1-\Gamma_2+\Gamma_{12})^2 dyds.
    \end{align*}
    We will now show that all the integrals above are $\lesssim \varepsilon_1\varepsilon_2$ with the constant possibly depending on $T$, which will conclude the proof.
    For $\Tilde{I}_{12}$, $1-\infty$ H\"older with the boundedness of the integrand and the area of $\Tilde{R}_{12}$ gives the result.\\
    For $I_1$, we first apply $1-\infty$ H\"older, giving us \begin{align*}
        I_1&\leq \varepsilon_2 T\max_{s\in [0,t-\frac{\varepsilon_1}{\sqrt{2}}]}\left|e^{-\frac{a}{2}(t-s)}-e^{-\frac{a}{2}(t-\frac{\varepsilon_1}{\sqrt{2}}-s)}\right|^2\\
        &\lesssim \varepsilon_2 \left|1-e^{-\frac{a\varepsilon_1}{2\sqrt{2}}}\right|
        \lesssim \varepsilon_1 \varepsilon_2.
    \end{align*}
    A similar argument works for $I_2$. This leaves us with $I_{12}$, for which we have \begin{align*}
        I_{12} &\leq T^2 \max_{s\in [0,t-\frac{\varepsilon_1+\varepsilon_2}{\sqrt{2}}]}\left|e^{-\frac{a}{2}(t-s)}+e^{-\frac{a}{2}(t-\frac{\varepsilon_1+\varepsilon_2}{\sqrt{2}}-s)}-(e^{-\frac{a}{2}(t-\frac{\varepsilon_1}{\sqrt{2}}-s)}+e^{-\frac{a}{2}(t-\frac{\varepsilon_2}{\sqrt{2}}-s)})\right|\\
        &\lesssim \left|1+e^{\frac{a}{2}(\frac{\varepsilon_1+\varepsilon_2}{\sqrt{2}})}-(e^{\frac{a}{2}(\frac{\varepsilon_1}{\sqrt{2}})}+e^{\frac{a}{2}(\frac{\varepsilon_2}{\sqrt{2}})})\right|.
    \end{align*}
    If one Taylor expands the exponential function, one can easily see that the term inside the last absolute value is $\frac{a^2}{2}\varepsilon_2\varepsilon_2+O(\varepsilon_1\varepsilon_2)$. This finishes the proof.
\end{proof}

\begin{proof}[Proof of Theorem \ref{th:sim:lil}]
    Fix $N>0$ and $0<a_1<a_2$.
    It suffices to show that the existence of $0<K_2<\infty$ such that for all $z_0 \in [a_1,a_2]$,
    \begin{align*}
        \P\left\{ \limsup_{h\to0^+} \frac{|u\big(\frac{w+z_0}{\sqrt{2}}+h, \frac{-w+z_0}{\sqrt{2}}+h\big)-u\big(\frac{w+z_0}{\sqrt{2}}, \frac{-w+z_0}{\sqrt{2}}\big)|}{\sqrt{h\log\log(1/h)}} \le K_2 \text{ for all $w \in [0,N]$} \right\} = 1.
    \end{align*}
    Let $0<K<\infty$ be a number whose value will be chosen later.
    For each integer $n \ge 2$, consider the event
    \begin{align*}
        A_n = \left\{ \sup_{w \in [0,b]}\sup_{h \in [2^{-n-1},2^{-n}]}\frac{|u\big(\frac{w+z_0+h}{\sqrt{2}}, \frac{-w+z_0+h}{\sqrt{2}}\big)-u\big(\frac{w+z_0}{\sqrt{2}},\frac{-w+z_0}{\sqrt{2}}\big)|}{\sqrt{h \log\log(1/h)}} > K \right\}.
    \end{align*}
    For simplicity, write
    \[
        \tilde{u}(w,z) = u\big(\tfrac{w+z}{\sqrt{2}}, \tfrac{-w+z}{\sqrt{2}}\big).
    \]
    Define $I_n=[0,N]\times[2^{-n-1},2^{-n}]$ and consider the process
    \[
        X(w,h)=\tilde{u}(w,z_0+h)-\tilde{u}(w,z_0), \quad 
        (w,h) \in I_n.
    \]
    Define the metric $d_X$ on $I_n$ by 
    \[
        d_X((w,h),(w',h')) = \|X(w,h)-X(w',h')\|_2 = (\E|X(w,h)-X(w',h')|^2)^{1/2}.
    \]
    Let $N(I_n, r)$ be the entropy number, i.e., the smallest number of $d_X$-balls of radius $r$ needed to cover $I_n$.
    Suppose $h'<h$. Then, by triangle inequality, Proposition \ref{lem:u-u}, and Lemma \ref{lem:rec_incr},
    \begin{align*}
        &d_X((w,h),(w',h'))
        = \|\tilde{u}(w,z_0+h) - \tilde{u}(w,z_0) - \tilde{u}(w',z_0+h') + \tilde{u}(w',z_0)\|_2\\
        &\le \|\tilde{u}(w,z_0+h)-\tilde{u}(w,z_0+h')\|_2 + \|\tilde{u}(w,z_0+h')-\tilde{u}(w,z_0) - \tilde{u}(w',z_0+h') + \tilde{u}(w',z_0) \|_2\\
        & \lesssim \sqrt{|h-h'|} + \sqrt{h'|w-w'|}
        \le \sqrt{|h-h'|} + \sqrt{2^{-n}|w-w'|}.
    \end{align*}
    In particular, this implies that the $d_X$-diameter $D_n$ of $I_n$ satisfies
    \[
        D_n \lesssim 2^{-n/2}
    \]
    and the entropy number satisfies
    \[
        N(I_n, r) \lesssim \frac{N2^{-n-1}}{2^n r^4} \lesssim \left( \frac{1}{2^{n/2}r} \right)^4.
    \]
    Thank to Borell's inequality \cite{Borell}, we know that for all $v>0$,
    \begin{align*}
        \P\left\{ \sup_{(w,h)\in I_n}|X(w,h)| - \E\left[ \sup_{(w,h) \in I_n}|X(w,h)| \right] > v\right\}
        \le \exp\left( - \frac{v^2}{2 \sup_{(w,h) \in I_n}\E[|X(w,h)|^2]} \right).
    \end{align*}    
    In particular, by Dudley's entropy theorem \cite{Dudley}, there exists a universal constant $C_1$ such that
    \begin{align*}
        &\E\left[ \sup_{(w,h) \in I_n}|X(w,h)| \right] 
        \le C_1 \int_0^{D_n} \sqrt{\log N(I_n,r)} dr\\
        &\le 4C_1 \int_0^{C_2 2^{-n/2}} \sqrt{\log(\frac{C_2}{2^{n/2}r})}dr
        = 4C_1 \int_0^\infty 2 C_2 2^{-n/2} x^2 e^{-x^2} dx
        \le C_3 2^{-n/2}
    \end{align*}
    for some constants $C_2, C_3$ that do not depend on $n$.
    Moreover, by Proposition \ref{lem:u-u}, there exists $C_4 > 0$ such that for all $n \ge 2$,
    \begin{align*}
        \sup_{(w,h) \in I_n} \E[|X(w,h)|^2] \le C_4 2^{-n}.
    \end{align*}
    Therefore, if $n$ is sufficiently large so that $\frac12 K \sqrt{2^{-n-1}\log\log 2^n} > C_3 2^{-n/2}$, then
    \begin{align*}
        \P(A_n) &\le \P\left\{ \sup_{(w,h)\in I_n}|X(w,h)| > K \sqrt{2^{-n-1} \log\log 2^n} \right\}\\
        & \le \P\left\{ \sup_{(w,h)\in I_n}|X(w,h)| -  \E\left[ \sup_{(w,h) \in I_n}|X(w,h)| \right] > \frac12 K \sqrt{2^{-n-1} \log\log 2^n} \right\}\\
        &\le \exp\left( - \frac{\frac14 K^2 2^{-n-1}\log\log2^n}{2 \sup_{(w,h) \in I_n}\E[|X(w,h)|^2]} \right)
        \le \exp\left( - \frac{K^2 \log\log2^n}{16C_4} \right).
    \end{align*}
    This ensures the existence of some large constant $K<\infty$ such that $\sum_{n=2}^\infty \P(A_n) < \infty$.
    This together with the Borel-Cantelli lemma implies the desired result.
\end{proof}

\section{Proof of Theorem \ref{th:sing}}
\label{s:6}

    Fix $0 \le w_0 < w_0'$, $0<z_0<z_0'$, and let $t_0 = w_0/\sqrt{2}$.
    We decompose $u$ as
    \[
        u(t,x) = u_1(t,x) + u_2(t,x)
    \] 
    for any $(t,x) \in [t_0,\infty) \times \R$, where
    \begin{align*}
        &u_1(t,x) = \int_0^{t_0} \int_\R \Gamma(t-s,x-y) W(ds,dy) = \frac12 \int_0^{t_0} \int_{x-(t-s)}^{x+(t-s)} e^{-a(t-s)/2} W(ds,dy),\\
        &u_2(t,x) = \int_{t_0}^t  \int_\R \Gamma(t-s,x-y) W(ds,dy) = \frac12 \int_{t_0}^t \int_{x-(t-s)}^{x+(t-s)} e^{-a(t-s)/2} W(ds,dy).
    \end{align*}
    Note that $u_1$ and $u_2$ are independent processes and $u_2$ is independent of $\mathscr{F}_{t_0}$.
    Let $w \in [w_0, w_0']$, $z \in [z_0,z_0']$, and write $(t,x) = (\tfrac{w+z}{\sqrt{2}}, \tfrac{-w+z}{\sqrt{2}})$.
    In order to analyze the existence and propagation of singularities, we write
    \begin{align}
        \notag&u_1(t+h,x+h)-u_1(t,x)\\
        \notag&= u_1\left(\tfrac{w+z}{\sqrt{2}}+h, \tfrac{-w+z}{\sqrt{2}}+h\right) - u_1\left(\tfrac{w+z}{\sqrt{2}}, \tfrac{-w+z}{\sqrt{2}}\right)\\
        \notag&= \frac12 \int_0^{t_0} \int_{x-(t-s)}^{x+(t-s)} (e^{-a(t+h-s)/2} - e^{-a(t-s)/2}) W(ds, dy) 
        + \frac12 \int_0^{t_0} \int_{x+(t-s)}^{x+(t-s)+2h} e^{-a(t+h-s)/2} W(ds, dy)\\
        \notag&= \frac{e^{-at/2}(e^{-ah/2}-1)}{2} \int_0^{t_0} \int_{-\sqrt{2}w+s}^{\sqrt{2}z-s} e^{as/2} W(ds, dy) 
        + \frac{e^{-a(t+h)/2}}{2} \int_0^{t_0} \int_{\sqrt{2}z-s}^{\sqrt{2}z-s+2h} e^{as/2} W(ds, dy)\\
        \label{E:u_1-u_1}&= \frac{e^{-\frac{a(w+z)}{2\sqrt{2}}}(e^{-ah/2}-1)}{2} X(w,z) + \frac{e^{-\frac{a(w+z)}{2\sqrt{2}}}e^{-ah/2}}{2} (Y(z+\sqrt{2} h)-Y(z)),
    \end{align}
    where
    \begin{align}\label{E:X:Y}
        X(w,z) = \int_0^{t_0} \int_{-\sqrt{2}w+s}^{\sqrt{2}z-s} e^{as/2} W(ds, dy), \qquad
        Y(z) = \int_0^{t_0} \int_{-s}^{\sqrt{2}z-s} e^{as/2} W(ds,dy)
    \end{align}

    We first establish some basic properties for the processes $X$ and $Y$.
    \begin{lemma}\label{lem:X}
        $X=\{X(w,z)\}_{(w,z) \in [w_0,w_0']\times [z_0,z_0']}$ is a centered Gaussian random field with 
        \begin{align}\label{E:X-X}
            \Var(X(w,z)-X(w',z')) \lesssim |w-w'|+|z-z'|
        \end{align}
        uniformly for all $(w,z), (w',z') \in [w_0,w_0']\times [z_0,z_0']$.
        Hence $X$ is a.s. continuous on $[w_0,w_0']\times [z_0,z_0']$.
    \end{lemma}

    \begin{proof}
        Consider $w>w'$ in $[w_0,w_0']$ and $z,z'$ in $[z_0,z_0']$.
        Assuming that $z>z'$, we have
        \begin{align*}
            &\Var(X(w,z)-X(w',z'))\\
            & \le 2 \Var(X(w,z)-X(w,z')) + 2 \Var(X(w,z')-X(w',z'))\\
            & = 2 \int_0^{t_0} \int_{\sqrt{2}z'-s}^{\sqrt{2}z-s} e^{as} dyds + 2 \int_0^{t_0} \int_{-\sqrt{2}w+s}^{-\sqrt{2}w'+s} e^{as} dyds\\
            & = 2\sqrt{2} (z-z') \int_0^{t_0} e^{as} ds + 2\sqrt{2} (w-w') \int_0^{t_0} e^{as} ds
            \lesssim |z-z'| + |w-w'|.
        \end{align*}
        The case of $z<z'$ is similar.
        This proves \eqref{E:X-X}. 
        The continuity of $X$ follows from a standard application of the Kolmogorov continuity theorem.
    \end{proof}

    \begin{lemma}\label{lem:Y}
        $\{ C_0^{-1} Y(z)\}_{z \ge 0}$ is a standard Brownian motion, where $C_0 = \sqrt{2} \int_0^{t_0} e^{as} ds$.
    \end{lemma}

    \begin{proof}
        For any $0\le z < z'$, by It\^o-Walsh isometry,
        \begin{align*}
            \Var(Y(z')-Y(z))
            = \int_0^{t_0} \int_{\sqrt{2}z-s}^{\sqrt{2}z'-s} e^{as} dyds = \sqrt{2}(z'-z) \int_0^{t_0} e^{as}ds = C_0(z'-z).
        \end{align*}
        By the Kolmogorov continuity theorem, $Y$ is a.s. continuous.
        This shows that $B(z):=C_0^{-1}Y(z)$ is a centered, continuous Gaussian process with $B(0) = 0$ and $\Var(B(z')-B(z))=|z'-z|$, and hence is a standard Brownian motion.
    \end{proof}
    
    We are now ready to prove Theorem \ref{th:sing}.
    As before, in order to simplify notations, we write
    % \begin{align*}
    %     \tilde{u}(w,z) &= u\big(\tfrac{w+z}{\sqrt{2}}, \tfrac{-w+z}{\sqrt{2}}\big),\\
    %     \tilde{u}_i(w,z) &= u_i\big(\tfrac{w+z}{\sqrt{2}}, \tfrac{-w+z}{\sqrt{2}}\big), \quad i=1,2,
    % \end{align*}
    % and
    \begin{align*}
        &\ell(w,z) = \limsup_{h\to0^+}\frac{|u\big(\frac{w+z}{\sqrt{2}}+h, \frac{-w+z}{\sqrt{2}}+h\big) - u\big(\frac{w+z}{\sqrt{2}}, \frac{-w+z}{\sqrt{2}}\big)|}{\sqrt{h \log\log(1/h)}},\\
        &\ell_i(w,z) = \limsup_{h\to0^+}\frac{|u_i\big(\frac{w+z}{\sqrt{2}}+h, \frac{-w+z}{\sqrt{2}}+h\big) - u_i\big(\frac{w+z}{\sqrt{2}}, \frac{-w+z}{\sqrt{2}}\big)|}{\sqrt{h \log\log(1/h)}}, \quad i=1,2.
    \end{align*}

\begin{proof}[Proof of Theorem \ref{th:sing}]
    Fix an integer $N\in \N_+$, and take $w_0' = w_0+N$.
    The processes $X$ and $Y$ defined in \eqref{E:X:Y} are both $\mathscr{F}_{t_0}$-measurable.
    By Lemma \ref{lem:X},
    \begin{align*}
        \P\left\{ \sup_{w \in [w_0, w_0+N]} \sup_{z \in [z_0,z_0']} |X(w,z)| < +\infty \right\} = 1.
    \end{align*}
    Since $|1-e^{-ah/2}| \lesssim h$ for $h>0$ small, the preceding implies that
    \begin{align}\label{X:lil}
        \P\left\{ \limsup_{h \to 0^+} \sup_{(w,z) \in [w_0,w_0+N]\times[z_0,z_0']} \frac{e^{-\frac{a(w+z)}{2\sqrt{2}}}|e^{-ah/2}-1||X(w,z)|}{\sqrt{h\log\log(1/h)}} = 0 \right\} = 1.
    \end{align}
    According to Lemma \ref{lem:Y}, $C_0^{-1}Y(z)$ is a standard Brownian motion.
    Hence, L\'evy's modulus of continuity theorem implies that
    \begin{align}\label{Y:mc}
        \limsup_{h \to 0^+} \sup_{z \in [z_0,z_0']}\frac{|Y(z+\sqrt{2}h)-Y(z)|}{\sqrt{h \log(1/h)}} = C_0 2^{3/4} \quad \text{a.s.}
    \end{align}
    With \eqref{Y:mc}, we may use Meyer's section theorem and an argument with nested intervals as in \cite{LeeXiao22} to obtain an $\mathscr{F}_{t_0}$-measurable random variable $Z$ such that $Z \in [z_0,z_0']$ a.s. and satisfies
    \begin{align}\label{Y:sing}
        \P\left\{ \limsup_{h \to 0^+} \frac{|Y(Z+\sqrt{2}h)-Y(Z)|}{\sqrt{h \log \log(1/h)}} = +\infty \right\} = 1.
    \end{align}
    By \eqref{E:u_1-u_1} and \eqref{limsup:ineq}, for all $w \in [w_0, w_0+N]$,
    \begin{align}\begin{split}\label{u_1:limsup:ineq}
        &\limsup_{h \to 0^+} \frac{|u_1\big(\frac{w+Z}{\sqrt{2}}+h, \frac{-w+Z}{\sqrt{2}}+h\big)-u_1\big(\frac{w+Z}{\sqrt{2}}, \frac{-w+Z}{\sqrt{2}}\big)|}{\sqrt{h \log\log(1/h)}}\\
        & \ge e^{-\frac{a(w+Z)}{2\sqrt{2}}} \limsup_{h\to0^+}\frac{e^{-ah/2}|Y(Z+\sqrt{2}h)-Y(Z)|}{2\sqrt{h \log\log(1/h)}} - \limsup_{h\to0^+} \frac{e^{-\frac{a(w+Z)}{2\sqrt{2}}}|e^{-ah/2}-1||X(w,Z)|}{2\sqrt{h\log\log(1/h)}}.
    \end{split}\end{align}
    In particular, the above holds for $w=w_0$.
    Hence, we may apply \eqref{X:lil} and \eqref{Y:sing} to deduce that
    \begin{align}\label{u_1:sing}
        % \P\left\{ \limsup_{h\to0^+} \frac{|u_1\big(\frac{w+Z}{\sqrt{2}}+h, \frac{-w+Z}{\sqrt{2}}+h\big)-u_1\big(\frac{w+Z}{\sqrt{2}}, \frac{-w+Z}{\sqrt{2}}\big)|}{\sqrt{h \log\log(1/h)}} = +\infty \right\} = 1.
        \P\left\{ \ell_1(w_0,Z) = +\infty \right\} = 1.
    \end{align}
    On the other hand, observe that 
    \begin{align}\label{E:u_2=u:law}
        \{u_2(t_0+t,x)\}_{(t,x) \in [0,\infty) \times \R} \text{ has the same law as }
        \{u(t,x)\}_{(t,x) \in [0,\infty)\times \R}.
    \end{align} 
    Indeed, since both of them are centered Gaussian random fields, this can be verified easily by comparing covariance:
    \begin{align*}
        \E[u_2(t_0+t,x)u_2(t_0+t',x')] =\E[u(t,x)u(t',x')]
        \quad \text{for all $(t,x),(t',x')\in [0,\infty) \times \R$.}
    \end{align*}
    Hence, it follows from Theorem \ref{th:sim:lil} that there exists $K_2<\infty$ such that for all fixed $z \in [z_0,z_0']$,
    \begin{align*}
        \P\left\{\ell_2(w_0,z) \le  K_2 \right\} 
        = \P\left\{ \ell(0,z) \le K_2 \right\}=1.
    \end{align*}
    But since $u_2$ is independent of $\mathscr{F}_{t_0}$ and the random variable $Z$ is $\mathscr{F}_{t_0}$-measurable, the preceding and conditioning yield
    \begin{align}\label{E:u_2}
        \P\left\{ \ell_2(w_0,Z) \le K_2 \right\}
        = \E\left[ \P\left\{ \ell_2(w_0,Z) \le K_2 \,\big|\, Z \right\} \right]
        = 1.
        % \limsup_{h\to0^+} \frac{|u_2\big(\frac{w_0+Z}{\sqrt{2}}+h, \frac{-w_0+Z}{\sqrt{2}}+h\big)-u_2\big(\frac{w_0+Z}{\sqrt{2}}, \frac{-w_0+Z}{\sqrt{2}}\big)|}{\sqrt{h\log\log(1/h)}} \le K_2\sqrt{w_0+Z}<\infty \quad \text{a.s.}
    \end{align}
    Thanks to \eqref{limsup:ineq}, we may combine \eqref{u_1:sing} and \eqref{E:u_2} to obtain
    \begin{align*}
        \ell(w_0, Z) \ge \ell_1(w_0,Z) - \ell_2(w_0,Z) \ge + \infty \quad \text{a.s.}
        % &\limsup_{h\to0^+} \frac{|u\big(\frac{w_0+Z}{\sqrt{2}}+h, \frac{-w_0+Z}{\sqrt{2}}+h\big)-u\big(\frac{w_0+Z}{\sqrt{2}}, \frac{-w_0+Z}{\sqrt{2}}\big)|}{\sqrt{h\log\log(1/h)}} \\
        % &\ge \limsup_{h\to0^+} \frac{|u_1\big(\frac{w_0+Z}{\sqrt{2}}+h, \frac{-w_0+Z}{\sqrt{2}}+h\big)-u_1\big(\frac{w_0+Z}{\sqrt{2}}, \frac{-w_0+Z}{\sqrt{2}}\big)|}{\sqrt{h\log\log(1/h)}}
        % - K_2 \sqrt{w_0+Z}\\
        % & = +\infty \quad \text{a.s.}
    \end{align*}
    This proves part (i) of Theorem \ref{th:sing}.
    To prove part (ii), let $Z$ be an $\mathscr{F}_{t_0}$-measurable random variable such that
    \begin{align}\label{Assump:Z}
        \P\left\{ Z \in [z_0,z_0'] \text{ and } \ell(w_0,Z) = +\infty \right\} = 1.
    \end{align}
    By Theorem \ref{th:sim:lil} and \eqref{E:u_2=u:law}, there exists $K_2<\infty$ such that for any fixed $z\in [z_0,z_0']$,
    \begin{align}\label{E:u_2:sim:lil}
        \P\left\{ \ell_2(w,z) \le K_2  \text{ for all $w\in [w_0,w_0+N]$} \right\}
        = \P\left\{ \ell(w,z) \le K_2 \text{ for all $w\in [0,N]$} \right\} = 1.
    \end{align}
    Since $u_2$ is independent of $\mathscr{F}_{t_0}$ and $Z$ is $\mathscr{F}_{t_0}$-measurable, we may use conditioning and apply \eqref{E:u_2:sim:lil} to deduce that
    \begin{align}\begin{split}\label{E:u_2:sim:lil:Z}
        &\P\left\{ \ell_2(w,Z) \le K_2  \text{ for all $w\in [w_0,w_0+N]$}  \right\}\\
        &= \E\left[ \P\left\{ \ell_2(w,Z) \le K_2 \text{ for all $w\in [w_0,w_0+N]$} \, \big|\, Z \right\} \right] = 1.
    \end{split}\end{align}
    In particular, \eqref{Assump:Z} and \eqref{E:u_2:sim:lil:Z} implies that
    \begin{align}\label{u_1:sing:w_0}
        \P\left\{ \ell_1(w_0,Z) = +\infty \right\}=1.
    \end{align}
    Now, we apply \eqref{E:u_1-u_1} and \eqref{limsup:ineq} to obtain the reverse inequality to \eqref{u_1:limsup:ineq}: For all $w \in [w_0,w_0+N]$,
    \begin{align*}
        &e^{-\frac{a(w+Z)}{2\sqrt{2}}} \limsup_{h\to0^+}\frac{e^{-ah/2}|Y(Z+\sqrt{2}h)-Y(Z)|}{2\sqrt{h \log\log(1/h)}} \\
        & \ge \limsup_{h \to 0^+} \frac{|u_1\big(\frac{w+Z}{\sqrt{2}}+h, \frac{-w+Z}{\sqrt{2}}+h\big)-u_1\big(\frac{w+Z}{\sqrt{2}}, \frac{-w+Z}{\sqrt{2}}\big)|}{\sqrt{h \log\log(1/h)}} - \limsup_{h\to0^+} \frac{e^{-\frac{a(w+Z)}{2\sqrt{2}}}|e^{-ah/2}-1||X(w,Z)|}{2\sqrt{h\log\log(1/h)}}.
    \end{align*}
    Thanks to \eqref{X:lil} and \eqref{u_1:sing:w_0}, we have
    \begin{align*}
        \P\left\{ e^{-\frac{a(w_0+Z)}{2\sqrt{2}}} \limsup_{h\to0^+}\frac{e^{-ah/2}|Y(Z+\sqrt{2}h)-Y(Z)|}{2\sqrt{h \log\log(1/h)}} = +\infty \right\} = 1
    \end{align*}
    and hence
    \begin{align*}
        \P\left\{ e^{-\frac{a(w+Z)}{2\sqrt{2}}}\limsup_{h\to0^+}\frac{e^{-ah/2}|Y(Z+\sqrt{2}h)-Y(Z)|}{2\sqrt{h \log\log(1/h)}} = +\infty \text{ for all $w\in [w_0, N]$} \right\} = 1.
    \end{align*}
    The preceding together with \eqref{u_1:limsup:ineq} and \eqref{X:lil} implies that
    \begin{align*}
        \P\left\{ \ell_1(w,Z) = +\infty \text{ for all $w \in [w_0,w_0+N]$} \right\} = 1.
    \end{align*}
    Combining this with \eqref{E:u_2:sim:lil:Z} yields
    \begin{align*}
        \P\left\{ \ell(w,Z) = +\infty \text{ for all $w \in [w_0,w_0+N]$} \right\} = 1.
    \end{align*}
    Since the integer $N$ can be arbitrarily large, this completes the proof.
\end{proof}

{\bf Acknowledgments.}
H. Chen is supported by the research grant (VIL73729) from Villum Fonden.
C.Y. Lee is supported in part by the Shenzhen Peacock fund 2025TC0013.
The authors would like to thank the anonymous referees for their helpful comments and suggestions.

\printbibliography

@article{LeeXiao22,
  title={Propagation of singularities for the stochastic wave equation},
  author={Lee, Cheuk Yin and Xiao, Yimin},
  journal={Stochastic Processes and their Applications},
  volume={143},
  pages={31--54},
  year={2022},
  publisher={Elsevier}
}

@book{evansPDE,
  title={Partial differential equations},
  author={Evans, Lawrence C},
  volume={19},
  year={2022},
  publisher={American mathematical society}
}

@article{CN88,
  title={Random nonlinear wave equations: propagation of singularities},
  author={Carmona, Ren{\'e} and Nualart, David},
  journal={The Annals of Probability},
  volume={16},
  number={2},
  pages={730--751},
  year={1988},
  publisher={Institute of Mathematical Statistics}
}

@article{DMX_Polarity,
  title={Polarity of points for Gaussian random fields},
  author={Robert C. Dalang and Carl Mueller and Yimin Xiao},
  journal={Annals of Probability},
  year={2017},
  volume={45},
  pages={4700-4751}
}

@article{Walsh82,
  title={Propagation of singularities in the Brownian sheet},
  author={Walsh, John B},
  journal={The Annals of Probability},
  volume={10},
  number={2},
  pages={279--288},
  year={1982},
  publisher={Institute of Mathematical Statistics}
}

@article{cerrai2006smoluchowski,
  title={On the Smoluchowski-Kramers approximation for a system with an infinite number of degrees of freedom},
  author={Cerrai, Sandra and Freidlin, Mark},
  journal={Probability theory and related fields},
  volume={135},
  number={3},
  pages={363--394},
  year={2006},
  publisher={Springer}
}

@article{cerrai2016smoluchowski,
  title={Smoluchowski--Kramers approximation and large deviations for infinite-dimensional nongradient systems with applications to the exit problem},
  author={Cerrai, Sandra and Salins, Michael},
  journal={Annals of probability: An official journal of the Institute of Mathematical Statistics},
  volume={44},
  number={4},
  pages={2591--2642},
  year={2016},
  publisher={Institute of Mathematical Statistics}
}

@article{cerrai2023small,
  title={On the small noise limit in the Smoluchowski-Kramers approximation of nonlinear wave equations with variable friction},
  author={Cerrai, Sandra and Xie, Mengzi},
  journal={Transactions of the American Mathematical Society},
  volume={376},
  number={11},
  pages={7651--7689},
  year={2023}
}

@article{khintchine1924satz,
  title={{\"u}ber einen satz der wahrscheinlichkeitsrechnung},
  author={Khintchine, Aleksandr},
  journal={Fundamenta Mathematicae},
  volume={6},
  number={1},
  pages={9--20},
  year={1924},
  publisher={Polska Akademia Nauk. Instytut Matematyczny PAN}
}

@incollection {Walsh86,
    AUTHOR = {Walsh, John B.},
     TITLE = {An introduction to stochastic partial differential equations},
 BOOKTITLE = {\'Ecole d'\'et\'e{} de probabilit\'es de {S}aint-{F}lour,
              {XIV}---1984},
    SERIES = {Lecture Notes in Math.},
    VOLUME = {1180},
     PAGES = {265--439},
 PUBLISHER = {Springer, Berlin},
      YEAR = {1986},
      ISBN = {3-540-16441-3},
   MRCLASS = {60H15 (35R60 60G20 60J80)},
  MRNUMBER = {876085},
MRREVIEWER = {Luis\ G.\ Gorostiza},
       DOI = {10.1007/BFb0074920},
       URL = {https://doi.org/10.1007/BFb0074920},
}

@article{brzezniak2005stochastic,
  title={Stochastic nonlinear beam equations},
  author={Brze{\'z}niak, Zdzis{\l}aw and Maslowski, Bohdan and Seidler, Jan},
  journal={Probability theory and related fields},
  volume={132},
  number={1},
  pages={119--149},
  year={2005},
  publisher={Springer}
}

@article{song2020stochastic,
  title={Stochastic conformal schemes for damped stochastic Klein-Gordon equation with additive noise},
  author={Song, Mingzhan and Qian, Xu and Shen, Tianlong and Song, Songhe},
  journal={Journal of Computational Physics},
  volume={411},
  pages={109300},
  year={2020},
  publisher={Elsevier}
}

@article{Hormander78,
  title={Propagation of singularities and semi-global existence theorems for (pseudo)-differential operators of principal type},
  author={H{\"o}rmander, Lars},
  journal={Annals of Mathematics},
  volume={108},
  number={3},
  pages={569--609},
  year={1978},
  publisher={JSTOR}
}

@article{hormander1971fourier,
  title={Fourier integral operators. I},
  author={H{\"o}rmander, Lars},
  journal={Acta Mathematica},
  volume={127},
  number={1},
  pages={79--183},
  year={1971},
  publisher={Springer}
}

@article{hormander1973existence,
  title={On the existence and the regularity of solutions of linear pseudodifferential equations},
  author={H{\"o}rmander, Lars Valter},
  journal={Uspekhi Matematicheskikh Nauk},
  volume={28},
  number={6},
  pages={109--164},
  year={1973},
  publisher={Russian Academy of Sciences, Steklov Mathematical Institute of Russian~…}
}

@article{sjostrand1982singularites,
  title={Singularit{\'e}s analytiques microlocales},
  author={Sj{\"o}strand, Johannes and Lascar, Bernard}, journal={Ast\'erisque},
  year={1982},
  publisher={Soci{\'e}t{\'e} math{\'e}matique de France}
}

@article{dencker1982propagation,
  title={On the propagation of polarization sets for systems of real principal type},
  author={Dencker, Nils},
  journal={Journal of Functional Analysis},
  volume={46},
  number={3},
  pages={351--372},
  year={1982},
  publisher={Elsevier}
}

@article{BesovWS,
  title={Besov wavefront set},
  author={Claudio Dappiaggi and Paolo Rinaldi and Federico Sclavi},
  journal={Analysis and Mathematical Physics},
  year={2022},
  volume={13}
}

@book{BCD11,
  title={Fourier analysis and nonlinear partial differential equations},
  author={Hajer Bahouri and Jean-Yves Chemin and Rapha{\"e}l Danchin},
  year={2011},
  publisher={Springer}
}

@article{hafner2021linear,
  title={Linear stability of slowly rotating Kerr black holes},
  author={H{\"a}fner, Dietrich and Hintz, Peter and Vasy, Andr{\'a}s},
  journal={Inventiones mathematicae},
  volume={223},
  number={3},
  pages={1227--1406},
  year={2021},
  publisher={Springer}
}

@article{hintz2018global,
  title={The global non-linear stability of the Kerr-de Sitter family of black holes.},
  author={Hintz, Peter and Vasy, Andr{\'a}s},
  journal={Acta Math.},
  volume={220},
  number={arXiv: 1606.04014},
  pages={1--206},
  year={2018}
}

@article{hintz2021normally,
  title={Normally hyperbolic trapping on asymptotically stationary spacetimes},
  author={Hintz, Peter},
  journal={Probability and Mathematical Physics},
  volume={2},
  number={1},
  pages={71--126},
  year={2021},
  publisher={Mathematical Sciences Publishers}
}

@article{sato1970regularity,
  title={Regularity of hyperfunction solutions of partial differential equations},
  author={Sato, Mikio},
  journal={Actes Congr. Int. Nat. Nice},
  volume={2},
  pages={785--794},
  year={1970}
}

@misc{vasy2011microlocalanalysisasymptoticallyhyperbolic,
      title={Microlocal analysis of asymptotically hyperbolic spaces and high energy resolvent estimates}, 
      author={Andras Vasy},
      year={2011},
      eprint={1104.1376},
      archivePrefix={arXiv},
      primaryClass={math.AP},
      url={https://arxiv.org/abs/1104.1376}, 
}

@book{hintz2025introduction,
  title={An Introduction to Microlocal Analysis},
  author={Hintz, P.},
  isbn={9783031907050},
  series={Graduate Texts in Mathematics},
  url={https://books.google.com.hk/books?id=Zes-0QEACAAJ},
  year={2025},
  publisher={Springer Nature Switzerland}
}

@article {LX23,
    AUTHOR = {Lee, Cheuk Yin and Xiao, Yimin},
     TITLE = {Chung-type law of the iterated logarithm and exact moduli of
              continuity for a class of anisotropic {G}aussian random
              fields},
   JOURNAL = {Bernoulli},
  FJOURNAL = {Bernoulli. Official Journal of the Bernoulli Society for
              Mathematical Statistics and Probability},
    VOLUME = {29},
      YEAR = {2023},
    NUMBER = {1},
     PAGES = {523--550},
      ISSN = {1350-7265,1573-9759},
   MRCLASS = {60G15 (60F15 60G17 60G60)},
  MRNUMBER = {4497257},
MRREVIEWER = {Ou\ Zhao},
       DOI = {10.3150/22-bej1467},
       URL = {https://doi.org/10.3150/22-bej1467},
}

@article {Dudley,
    AUTHOR = {Dudley, R. M.},
     TITLE = {The sizes of compact subsets of {H}ilbert space and continuity
              of {G}aussian processes},
   JOURNAL = {J. Functional Analysis},
  FJOURNAL = {Journal of Functional Analysis},
    VOLUME = {1},
      YEAR = {1967},
     PAGES = {290--330},
      ISSN = {0022-1236},
   MRCLASS = {60.40 (46.00)},
  MRNUMBER = {220340},
MRREVIEWER = {H.\ Heyer},
       DOI = {10.1016/0022-1236(67)90017-1},
       URL = {https://doi.org/10.1016/0022-1236(67)90017-1},
}

@article {Borell,
    AUTHOR = {Borell, Christer},
     TITLE = {The {B}runn-{M}inkowski inequality in {G}auss space},
   JOURNAL = {Invent. Math.},
  FJOURNAL = {Inventiones Mathematicae},
    VOLUME = {30},
      YEAR = {1975},
    NUMBER = {2},
     PAGES = {207--216},
      ISSN = {0020-9910,1432-1297},
   MRCLASS = {28A40 (60G15)},
  MRNUMBER = {399402},
MRREVIEWER = {A.\ Badrikian},
       DOI = {10.1007/BF01425510},
       URL = {https://doi.org/10.1007/BF01425510},
}

@article {Zim,
    AUTHOR = {Zimmerman, Grenith J.},
     TITLE = {Some sample function properties of the two-parameter
              {G}aussian process},
   JOURNAL = {Ann. Math. Statist.},
  FJOURNAL = {Annals of Mathematical Statistics},
    VOLUME = {43},
      YEAR = {1972},
     PAGES = {1235--1246},
      ISSN = {0003-4851},
   MRCLASS = {60G15 (60F15 60H05)},
  MRNUMBER = {317401},
MRREVIEWER = {H.\ Bergstr\"om},
       DOI = {10.1214/aoms/1177692475},
       URL = {https://doi.org/10.1214/aoms/1177692475},
}

@article {Orey-Pruitt,
    AUTHOR = {Orey, Steven and Pruitt, William E.},
     TITLE = {Sample functions of the {$N$}-parameter {W}iener process},
   JOURNAL = {Ann. Probability},
  FJOURNAL = {The Annals of Probability},
    VOLUME = {1},
      YEAR = {1973},
    NUMBER = {1},
     PAGES = {138--163},
      ISSN = {0091-1798},
   MRCLASS = {60J65},
  MRNUMBER = {346925},
       DOI = {10.1214/aop/1176997030},
       URL = {https://doi.org/10.1214/aop/1176997030},
}

@article {MWX,
    AUTHOR = {Meerschaert, Mark M. and Wang, Wensheng and Xiao, Yimin},
     TITLE = {Fernique-type inequalities and moduli of continuity for
              anisotropic {G}aussian random fields},
   JOURNAL = {Trans. Amer. Math. Soc.},
  FJOURNAL = {Transactions of the American Mathematical Society},
    VOLUME = {365},
      YEAR = {2013},
    NUMBER = {2},
     PAGES = {1081--1107},
      ISSN = {0002-9947,1088-6850},
   MRCLASS = {60G15 (60F15 60G17 60G60)},
  MRNUMBER = {2995384},
MRREVIEWER = {Lee-Peng\ Teo},
       DOI = {10.1090/S0002-9947-2012-05678-9},
       URL = {https://doi.org/10.1090/S0002-9947-2012-05678-9},
}

@article{Hu-Lee,
  title={On the spatio-temporal increments of nonlinear parabolic SPDEs and the open KPZ equation},
  author={Hu, Jingwu and Lee, Cheuk Yin},
  journal={arXiv preprint arXiv:2508.05032},
  year={2025}
}

@article {Ayache-Xiao,
    AUTHOR = {Ayache, Antoine and Xiao, Yimin},
     TITLE = {Asymptotic properties and {H}ausdorff dimensions of fractional
              {B}rownian sheets},
   JOURNAL = {J. Fourier Anal. Appl.},
  FJOURNAL = {The Journal of Fourier Analysis and Applications},
    VOLUME = {11},
      YEAR = {2005},
    NUMBER = {4},
     PAGES = {407--439},
      ISSN = {1069-5869,1531-5851},
   MRCLASS = {60G15 (28A80 42C40 60G20)},
  MRNUMBER = {2169474},
MRREVIEWER = {Werner\ Linde},
       DOI = {10.1007/s00041-005-4048-3},
       URL = {https://doi.org/10.1007/s00041-005-4048-3},
}

@inproceedings {Orey72,
    AUTHOR = {Orey, Steven},
     TITLE = {Growth rate of certain {G}aussian processes},
 BOOKTITLE = {Proceedings of the {S}ixth {B}erkeley {S}ymposium on
              {M}athematical {S}tatistics and {P}robability ({U}niv.
              {C}alifornia, {B}erkeley, {C}alif., 1970/1971), {V}ol. {II}:
              {P}robability theory},
     PAGES = {443--451},
 PUBLISHER = {Univ. California Press, Berkeley, CA},
      YEAR = {1972},
   MRCLASS = {60G15},
  MRNUMBER = {402897},
MRREVIEWER = {Simeon\ M.\ Berman},
}

@book{khintchine1933asymptotische,
  title={Asymptotische Gesetze der Wahrscheinlichkeitsrechnung},
  author={Khintchine, Aleksandr},
  year={1933},
  publisher={Springer}
}

@article{Levy,
  title={Th{\'e}orie de l’addition des variables al{\'e}atoires},
  author={L{\'e}vy, P},
  journal={Gauthier-Villars, Paris},
  year={1937}
}

@article {LX19,
    AUTHOR = {Lee, Cheuk Yin and Xiao, Yimin},
     TITLE = {Local nondeterminism and the exact modulus of continuity for
              stochastic wave equation},
   JOURNAL = {Electron. Commun. Probab.},
  FJOURNAL = {Electronic Communications in Probability},
    VOLUME = {24},
      YEAR = {2019},
     PAGES = {Paper No. 52, 8},
      ISSN = {1083-589X},
   MRCLASS = {60G15 (60G17 60H15)},
  MRNUMBER = {4003126},
MRREVIEWER = {Seiichiro\ Kusuoka},
       DOI = {10.1214/19-ecp264},
       URL = {https://doi.org/10.1214/19-ecp264},
}

@book {Marcus-Rosen,
    AUTHOR = {Marcus, Michael B. and Rosen, Jay},
     TITLE = {Markov processes, {G}aussian processes, and local times},
    SERIES = {Cambridge Studies in Advanced Mathematics},
    VOLUME = {100},
 PUBLISHER = {Cambridge University Press, Cambridge},
      YEAR = {2006},
     PAGES = {x+620},
      ISBN = {978-0-521-86300-1; 0-521-86300-7},
   MRCLASS = {60-02 (60G15 60J25 60J55)},
  MRNUMBER = {2250510},
MRREVIEWER = {Nathalie\ Eisenbaum},
       DOI = {10.1017/CBO9780511617997},
       URL = {https://doi.org/10.1017/CBO9780511617997},
}

@article {Orey-Taylor,
    AUTHOR = {Orey, Steven and Taylor, S. James},
     TITLE = {How often on a {B}rownian path does the law of iterated
              logarithm fail?},
   JOURNAL = {Proc. London Math. Soc. (3)},
  FJOURNAL = {Proceedings of the London Mathematical Society. Third Series},
    VOLUME = {28},
      YEAR = {1974},
     PAGES = {174--192},
      ISSN = {0024-6115,1460-244X},
   MRCLASS = {60J65 (60G17)},
  MRNUMBER = {359031},
MRREVIEWER = {P.\ W.\ Millar},
       DOI = {10.1112/plms/s3-28.1.174},
       URL = {https://doi.org/10.1112/plms/s3-28.1.174},
}

@incollection {Kh-Shi,
    AUTHOR = {Khoshnevisan, Davar and Shi, Zhan},
     TITLE = {Fast sets and points for fractional {B}rownian motion},
 BOOKTITLE = {S\'eminaire de {P}robabilit\'es, {XXXIV}},
    SERIES = {Lecture Notes in Math.},
    VOLUME = {1729},
     PAGES = {393--416},
 PUBLISHER = {Springer, Berlin},
      YEAR = {2000},
      ISBN = {3-540-67314-8},
   MRCLASS = {60G17 (60G15 60G18)},
  MRNUMBER = {1768077},
MRREVIEWER = {Laurent\ Decreusefond},
       DOI = {10.1007/BFb0103816},
       URL = {https://doi.org/10.1007/BFb0103816},
}

@article {KPX,
    AUTHOR = {Khoshnevisan, Davar and Peres, Yuval and Xiao, Yimin},
     TITLE = {Limsup random fractals},
   JOURNAL = {Electron. J. Probab.},
  FJOURNAL = {Electronic Journal of Probability},
    VOLUME = {5},
      YEAR = {2000},
     PAGES = {no. 5, 24},
      ISSN = {1083-6489},
   MRCLASS = {60J65 (28A80 60D05 60G17)},
  MRNUMBER = {1743726},
MRREVIEWER = {Norbert\ Patzschke},
       DOI = {10.1214/EJP.v5-60},
       URL = {https://doi.org/10.1214/EJP.v5-60},
}

@article {Huang-Kh,
    AUTHOR = {Huang, Jingyu and Khoshnevisan, Davar},
     TITLE = {On the multifractal local behavior of parabolic stochastic
              {PDE}s},
   JOURNAL = {Electron. Commun. Probab.},
  FJOURNAL = {Electronic Communications in Probability},
    VOLUME = {22},
      YEAR = {2017},
     PAGES = {Paper No. 49, 11},
      ISSN = {1083-589X},
   MRCLASS = {60H15 (35R60 60G17)},
  MRNUMBER = {3710805},
MRREVIEWER = {Paul\ Andr\'e\ Razafimandimby},
       DOI = {10.1214/17-ECP86},
       URL = {https://doi.org/10.1214/17-ECP86},
}

@article {Blath-Martin,
    AUTHOR = {Blath, Jochen and Martin, Andreas},
     TITLE = {Propagation of singularities in the semi-fractional {B}rownian
              sheet},
   JOURNAL = {Stochastic Process. Appl.},
  FJOURNAL = {Stochastic Processes and their Applications},
    VOLUME = {118},
      YEAR = {2008},
    NUMBER = {7},
     PAGES = {1264--1277},
      ISSN = {0304-4149,1879-209X},
   MRCLASS = {60G15 (60G17)},
  MRNUMBER = {2428718},
MRREVIEWER = {Yimin\ Xiao},
       DOI = {10.1016/j.spa.2007.06.010},
       URL = {https://doi.org/10.1016/j.spa.2007.06.010},
}
\end{document}